\documentclass[11pt, reqno]{amsart}  	
\usepackage{geometry}                		
\usepackage{graphicx}				
\usepackage{amssymb}
\usepackage{amsmath}
\usepackage{comment}
\usepackage{amsthm}
\usepackage{mathrsfs}
\allowdisplaybreaks
\newtheorem{theorem}{Theorem}
\numberwithin{theorem}{section}
\newtheorem{proposition}[theorem]{Proposition}
\newtheorem{lemma}[theorem]{Lemma}
\newtheorem{corollary}[theorem]{Corollary}

\usepackage[colorlinks=true, linkcolor=blue,
 citecolor=blue, urlcolor=blue]{hyperref}
 \usepackage{dsfont}
 \usepackage{tikz-cd}
 \usepackage{caption}
 \usepackage{enumitem} 
 \usepackage{mathdots}

\newcommand{\Aut}{{\mathrm{Aut}}}

\newcommand{\Ind}{{\mathrm{Ind}}}

\newcommand{\SL}{{\mathrm{SL}}}
\newcommand{\GL}{{\mathrm{GL}}}
\newcommand{\PGL}{{\mathrm{PGL}}}

\newcommand{\SO}{{\mathrm{SO}}}

\newcommand{\Hom}{{\mathrm{Hom}}}

\title{The Dual Pair $\mathrm{Aut}(C)\times F_{4}$ ($p$-adic case)}
\author{Edmund Karasiewicz and Gordan Savin}
\address{Edmund Karasiewicz: Department of Mathematics, University of Utah}
\address{Gordan Savin: Department of Mathematics, University of Utah}

\begin{document}
\begin{abstract}
We study the local theta correspondence for dual pairs of the form $\mathrm{Aut}(C)\times F_{4}$ over a $p$-adic field, where $C$ is a composition algebra of dimension $2$ or $4$, by restricting the minimal representation of a group of type $E$. We investigate this restriction through the computation of maximal parabolic Jacquet modules and the Fourier-Jacobi functor.

As a consequence of our results 
we  prove a multiplicity one result for the $\mathrm{Spin}(9)$-invariant linear functionals of irreducible representations of $F_{4}$ and classify 
 the $\mathrm{Spin}(9)$-distinguished representations. 
 \end{abstract}

\maketitle

\section{Introduction}

Let $F$ be a $p$-adic field, i.e. a nonarchimedean local field of characteristic $0$ and residual characteristic $p>0$. We study the local theta correspondence of the $F$-points of the dual pair $\mathrm{Aut}(C)\times F_{4}$, where $C$ is a composition $F$-algebra of dimension $2$ or $4$, by restricting the minimal representation $(\Pi,\mathcal{V})$ of a group of type $E$. For this introduction we specialize to $\mathrm{dim}(C)=4$, for simplicity. In this case, the group of type $E$ is the adjoint form of $E_{7,C}$, split when $C$ is split and the unique nonsplit form when $C$ is anisotropic. 

With a dual pair $\mathrm{Aut}(C)\times F_{4}\subset E_{7,C}$ one can lift representations from $\mathscr{G}=\mathrm{Aut}(C)(F)$ to $G=F_{4}(F)$ as follows. Given $\tau\in \mathrm{Irr}(\mathscr{G})$ a smooth irreducible representation of $\mathscr{G}$, the maximal $\tau$-isotypic quotient of $\mathcal{V}$ admits an action of $G$ and factors as $\tau\otimes \Theta(\tau)$, where $\Theta(\tau)$ is a smooth representation of $G$. The representation $\Theta(\tau)$ is called the big theta lift of $\tau$.
 Its maximal semisimple quotient $\theta(\tau)$ (co-socle) is called the small theta lift of $\tau$. Note that one may reverse the roles of $\mathscr{G}$ and $G$. The primary objective of this paper is to investigate the big and small theta lifts of the dual pair $\mathrm{Aut}(C)\times F_{4}\subset E_{7,C}$.

We begin by discussing the theta lift from $\mathrm{Aut}(C)$ to $F_{4}$. Our first theorem gives a qualitative behavior of the lift. 
It  is a combination of Theorems \ref{SuperCuspLift}, \ref{ThetaPS} and \ref{ThetaPSsubQ}. 

\begin{theorem}
Let $\tau\in \mathrm{Irr}(\mathrm{Aut}(C))$. Then: 
\begin{enumerate}
\item $\Theta(\tau)\neq 0$  and it is a finite length representation of $F_{4}$. 
\item If $\tau$ is tempered then $\Theta(\tau)$ is irreducible. 
\item If $\theta(\tau)\cong \theta(\tau')$, where $\tau'\in \mathrm{Irr}(\mathrm{Aut}(C))$, then $\tau\cong \tau'$.
\end{enumerate}
\end{theorem}

For lifting in the opposite direction, that is from $F_{4}$ to $\mathrm{Aut}(C)$, our main result is Theorem \ref{F4toPGL2Irr}. It says, if $\sigma\in\mathrm{Irr}(F_{4})$ such that $\Theta(\sigma)\neq 0$, then $\Theta(\sigma)\in \mathrm{Irr}(\mathrm{Aut}(C))$.

For our second theorem, we specialize to the case where $C$ is the algebra of $2\times 2$ matrices, so $\mathrm{Aut}(C)=\mathrm{PGL}_{2}$. In this case, we can 
completely describe $\Theta(\tau)$. In order to state the results, we note that the $F_4$ group in this paper is not realized as a Chevalley group but as 
 the group of automorphism of a 27-dimensional exceptional Jordan algebra $J$. Thus $F_4$ acts on 
  the 26-dimensional subspace $J^0$ of trace 0 elements in $J$ and its
maximal parabolic subgroups can be described  as stabilizers of singular subspaces of $J_0$ \cite{A87}. In particular, $F_4$ has maximal parabolic subgroups $Q$ and 
$Q_2$ stabilizing 1 and 2-dimensional singular spaces, respectively.  (We record that Levi subgroups of $Q$ and $Q_2$ have the type $B_3$ and 
$A_{2,\mathrm{long}} \times A_{1,\mathrm{short}}$, respectively.)   Observe that $Q$ and $Q_2$, via their actions on the stabilized 1 and 2-dimensional singular spaces, 
have quotients isomorphic to $\GL_1$ and $\GL_2$, respectively. 
In particular, a character $\chi$ of $\GL_1(F)$ defines a degenerate principal series representation $\Ind_Q^G(\chi)$, and 
a supercuspidal representation $\tau$ of $\mathrm{PGL}_{2}(F)$ defines a 
family of degenerate principal series representations $\Ind_{Q_2}^G(\tau \otimes |\det|^s)$.  

\begin{theorem}\label{IntroPGL2Theta} Let $\tau\in \mathrm{Irr}(\mathrm{PGL}_{2}(F))$.
\begin{enumerate}
\item If $\tau$ is a quotient of a principal series $\Ind_{\overline{\mathscr B}}^{\mathscr G}(\chi)$ then $\Theta(\tau)$ is a quotient of $\Ind_Q^G(\chi)$. 
(For the precise statement see Theorems \ref{ThetaPS} and \ref{ThetaPSsubQ} and Proposition \ref{SmallThetaPS}.) \label{PS}
\item If $\tau$ is a supercuspidal representation, then $\Theta(\tau)=\theta(\tau)$ is the unique irreducible quotient of 
 $\Ind_{Q_2}^G(\tau \otimes |\det|^{3/2})$.
\end{enumerate}
\end{theorem}

In (\ref{PS}) $\theta(\tau)$ is always isomorphic to the co-socle of the degenerate principal series. Since the co-socle of $\Ind_Q^G(|\cdot|^{5/2})$ 
is a sum of two irreducible representations the theta correspondence is not one to one, and this is the only place where it fails.

Next we want to highlight some consequences of our results, as they relate to the relative Langlands program of Sakellaridis-Venkatesh \cite{SV17}. 
We prove that that the rank one exceptional symmetric pair $(F_4, {\mathrm{Spin}_9})$ over a $p$-adic field is a Gelfand pair, a long time open problem:


\begin{theorem}\label{IntroApps}$ $
Let $\sigma\in\mathrm{Irr}(F_{4})$. Then $\mathrm{dim}\mathrm{Hom}_{\mathrm{Spin}(9)}(\tilde\sigma, \mathbb{C})\leq 1$. Moreover, the dimension is $1$ if and only if $\sigma$ is the theta lift of a generic representation of $\mathrm{PGL}_{2}(F)$. 

\end{theorem}

The study of symmetric spaces, and more generally spherical spaces, has a long history.
In \cite{vD86} van Dijk proved that real forms of the symmetric pair $(F_4, \mathrm{Spin}(9))$ 
are generalized Gelfand pairs, a slightly weaker statement, as it concerns unitary representations only. Recently Rubio \cite{RR19} proved that 
 $(F_4, \mathrm{Spin}(9))$ is a Gelfand pair over $\mathbb C$. 
The usual approach  involves invariant distributions, see \cite{Gr91} or  \cite{AG08} for more information on this rich subject. 
 On the other hand, Howe \cite{Ho79} used the dual pair  $\SL_2 \times \mathrm{O}(n)$ to analyze the symmetric pair $(\mathrm{O}(n), \mathrm{O}(n-1))$. 
  It was observed in \cite{S94} that Howe's strategy can be applied to all rank one symmetric pairs. In this paper, at long last, 
 we execute this strategy for the exceptional symmetric pair. 
 Theorem \ref{IntroApps} is a consequence of the fact that the theta correspondence relates the $\mathrm{Spin}(9)$-period on representations of $F_4$ to the Whittaker period on 
representations of $\mathrm{PGL}_2(F)$. More precisely, we have 
\[ 
\mathrm{Hom}_{\mathrm{Spin}(9)}(\tilde{\sigma}, \mathbb{C}) \cong \mathrm{Hom}_{\mathscr{U},\psi}(\Theta(\sigma), \mathbb{C}) 
\] 
where $(\mathscr{U},\psi)$ is  a Whittaker datum for $\mathrm{PGL}_2(F)$.  Since we proved that $\Theta(\sigma)$ is irreducible (or zero) Theorem \ref{IntroApps} follows from 
uniqueness of the Whittaker functional for irreducible representations of $\mathrm{PGL}_2(F)$.  Moreover, since the lift from $\mathrm{PGL}_{2}(F)$ is completely known by 
Theorem \ref{IntroPGL2Theta}, 
we have a classification of $\mathrm{Spin}(9)$-distinguished representations of $F_4$, consistent with predictions made in  \cite{SV17}.

The primary tools in our analysis are computations of maximal parabolic Jacquet modules and the Fourier-Jacobi functor. Similar Jacquet module computations were used to study several different exceptional dual pairs in \cite{GS99, GS23, MS97}. In particular, this type of Jacquet module computation provides an important step in establishing Howe duality and dichotomy for exceptional dual pairs containing $G_{2}$, which was recently completed in \cite{GS23}.

Our main new input is the use of the Fourier-Jacobi functor. This allows us to relate the $\mathrm{Aut}(C)\times F_{4}$ theta correspondence to a classical $\mathrm{O}(3)\times \mathrm{Sp}(6)$ theta correspondence. Using the well developed theory of this classical theta correspondence we can efficiently derive results about the $\mathrm{Aut}(C)\times F_{4}$ theta correspondence.

Now we outline the contents of the paper and make a few more remarks on the proofs of our main results. Section \ref{Notation} introduces notation and recalls some preliminary material. Section \ref{JMI} contains the computations of (twisted) Jacquet modules of the minimal representation $\mathcal{V}$ with respect to a maximal Heisenberg parabolic of $F_{4}$. These calculations are done using a filtration of $\mathcal{V}$ with respect to a maximal Heisenberg parabolic subgroup of $E_{7}$ (recalled in Theorem \ref{HeisJacThm}). This filtration was first studied in Magaard-Savin \cite{MS97}. In this section we also review the Fourier-Jacobi functor.

In Section \ref{E7SuperCusp} we apply the results of Section \ref{JMI} to study the theta lift of $\tau$ a supercuspidal representation of $\mathrm{Aut}(C)$ to $F_{4}$. The main result of this section Theorem \ref{SuperCuspLift} states that $\Theta(\tau)$ is irreducible. The proof is based on the Fourier-Jacobi functor, which is the main new input in our analysis. Its utility stems from Proposition \ref{FJMin}, which says that the Fourier-Jacobi functor applied to the minimal representation of $E_{7}$ is isomorphic to the Weil representation as an $\mathrm{SO}(3)\times \mathrm{Sp}(6)$-representation. This is almost the setting of the classical dual pair $\mathrm{O}(3)\times \mathrm{Sp}(6)$. We use the well developed classical theory and basic properties of the Fourier-Jacobi functor to deduce that $\Theta(\tau)$ has at most two nontrivial constituents (Corollary \ref{AtMost2}). Then we apply the calculations from Section \ref{JMI} to show that $\Theta(\tau)$ is irreducible. Specifically, we use the twisted Jacquet module calculations to prove that $\Theta(\tau)$ has at most one nontrivial constituent (Proposition \ref{UniqueNonTriv}), and the untwisted Jacquet module to rule out the trivial representation (Proposition \ref{NoTriv}).

Next we specialize to the case when $C$ is the algebra of $2\times 2$ matrices, so $\mathrm{Aut}(C)=\mathrm{PGL}_{2}$. Section \ref{JMII} is roughly analogous to Section \ref{JMI}. The difference is that now we use a filtration of $\mathcal{V}$ with respect to a maximal Siegel parabolic subgroup of $E_{7}$ \cite{S94} (recalled in Theorem \ref{PFiltration}) to compute Jacquet modules with respect to a Borel subgroup of $\mathrm{PGL}_{2}$.

In Section \ref{LiftPSeries} we describe the theta lift of representations of $\mathrm{PGL}_{2}$ to $F_{4}$. This breaks up into two parts. First we consider constituents of principal series. For this we apply the results of Section \ref{JMII} on untwisted Jacquet modules to lift the constituents of principal series of $\mathrm{PGL}_{2}$ to $F_{4}$. Generically, the theta lift of a $\mathrm{PGL}_{2}$ principal series is a degenerate principal series of $F_{4}$ induced from the maximal parabolic subgroup $Q$. The complete description of the big theta lift is contained in Theorems \ref{ThetaPS} and \ref{ThetaPSsubQ}; the small theta lift is described in Proposition \ref{SmallThetaPS}. The approach of this section builds upon \cite{S94}.

Second, we consider supercuspidal representations in Subsection \ref{PGL2SC}. From Theorem \ref{SuperCuspLift} we know that the theta lift of a supercuspdial representation is irreducible. Here we refine this result in Proposition \ref{ThetaSCasQuotient} when $\mathrm{Aut}(C)=\mathrm{PGL}_{2}$. Specifically, we show that the theta lift is a quotient of an explicit representation of $F_{4}$ induced from the maximal parabolic subgroup $Q_2$. 
We note that this calculation uses the $G_{2}\times F_{4}\subset E_{8}$ dual pair studied in Magaard-Savin \cite{MS97}.

In Section \ref{F4toAutC} we consider the theta lift from $F_{4}$ to $\mathrm{Aut}(C)$. The main result is Theorem \ref{F4toPGL2Irr}, which states that if $\sigma\in\mathrm{Irr}(F_{4})$ and $\Theta(\sigma)\neq 0$, then $\theta(\sigma)\in\mathrm{Irr}(\mathrm{Aut}(C))$.

In Section \ref{Spin9} we characterize the irreducible representations of $F_{4}$ that are $\mathrm{Spin}(9)$-distinguished, i.e. possess a $\mathrm{Spin}(9)$-invariant linear functional.  The main result, Theorem \ref{Spin9Period}, is proved using the twisted Jacquet module calculations from Section \ref{JMII}.

Section \ref{QuadCase} concludes the paper with analogous (but easier) results when $\mathrm{dim}C=2$.
\section{Notation}\label{Notation}

\subsection{Representation theory of $p$-adic groups}

Let $F$ be a nonarchimedean local field of characteristic $0$ and residual characteristic $p>0$, with ring of integers $\mathcal{O}$ and maximal ideal $\mathfrak{p}$. We fix a uniformizer $\varpi\in \mathfrak{p}$ and a nontrivial additive character $\psi:F\rightarrow \mathbb{C}^{\times}$. We normalize the absolute value on $F$ so that $|\varpi|=q^{-1}$.

Let $G$ be the $F$-points of a connected reductive group. Let $\mathcal{M}(G)$ be the category of smooth $G$-representations and let $\mathrm{Irr}(G)$ be the set of isomorphism classes of irreducible objects. Given $\pi\in\mathcal{M}(G)$ we write $\widetilde{\pi}$ for the smooth contragradient representation of $\pi$. 

Let $P=MN$ be a parabolic subgroup of $G$ with a Levi decomposition. We use the following notation for unnormalized parabolic induction. Let $(\sigma,W)$ be a smooth representation of $M$ inflated to $P$. We write $\mathrm{Ind}_{P}^{G}(\sigma)$ for the space of right $G$-smooth functions $f:G\rightarrow W$ such that for any $g\in G$, $m\in M$, and $n\in N$ we have $f(mng)=\sigma(m)f(g)$. This is a $G$-representation with the action $(g\cdot f)(g^{\prime})=f(g^{\prime}g)$.

We fix $dn$ to be a Haar measure for $N$ and write $\delta_{P}$ for the modular character of $P$ defined by $d(pnp^{-1})=\delta_{P}(p)dn$. We write $i_{M}^{G}(\sigma)=\mathrm{Ind}_{P}^{G}(\delta_{P}^{1/2}\otimes\sigma)$ for normalized parabolic induction.

Let $(\pi,V)$ be a smooth $G$-representation. If $H\subset G$ is a subgroup with a character $\chi:H\rightarrow \mathbb{C}$, let $V_{(H,\chi)}$ denote the space of $(H,\chi)$-coinvariants. This space can be realized as the quotient of $V$ by the subspace $\mathrm{span}\{h\cdot v-\chi(h)v|v\in V,\,h\in H\}$ and is a representation of the subgroup of the normalizer of $H$ that fixes $\chi$, which we write as $\mathrm{Stab}_{G}(\chi)$.

If $P=MN\subset G$ is a parabolic subgroup, we write $r_{P}(V)=\delta_{P}^{-1/2}\otimes V_{N}$ for the normalized Jacquet module.

\subsection{Composition Algebras}

The theta-lift examined in this paper is based on exceptional dual pairs that can be constructed using composition and Jordan algebras. We begin by collecting some information on these algebras.

Let $C$ be a composition algebra over $F$ with quadratic norm form $n_{C}$. We write $B_{C}(x,y)=n_{C}(x+y)-n_{C}(x)-n_{C}(y)$ for the bilinear form associated to $n_{C}$ and $\overline{\phantom{x}}:C\rightarrow  C$ as $x\mapsto \overline{x}$ for conjugation (\cite{SV00}, Section 1.3). The trace of an element of $x\in C$ is $\mathrm{Tr}_{C}(x)=x+\overline{x}$ and $n_{C}(x)=x\overline{x}$. Note that $\mathrm{Tr}_{C}(xy)=-B_{C}(x,y)$, for all $x,y\in C$.  

We write $C^{0}$ for the subspace of trace $0$ elements of $C$. The group of $F$-algebra automorphisms $\mathrm{Aut}(C)$ preserves the norm and acts on $C^{0}$. Thus $\mathrm{Aut}(C)$ is contained in $O(C^{0},n_{C})$ the orthogonal group of the norm form.

Recall that $\mathrm{dim}_{F}(C)=1,2,4,8$. Over the $p$-adic field $F$ the possible composition algebras can be described explicitly. When $\mathrm{dim}_{F}(C)=2$, then $C$ is either a quadratic field extension of $F$, or $C$ is isomorphic to the split quadratic algebra $F\oplus F$ with norm form $(x,y)\mapsto xy$. In either case, the automorphism group of $\mathrm{Aut}(C)$ is generated by the conjugate map $x\mapsto \overline{x}$ and so is isomorphic to $\mu_{2}=\{\pm1\}$.

When $\mathrm{dim}(C)=4$, $C$ is either isomorphic to the split quaternion algebra, which can be realized as the algebra of $2\times 2$ matrices $M(2,F)$ with norm form given by the determinant; or $C$ is isomorphic to $D$, the unique (up to isomorphism) quaternion division algebra over $F$. The group of $\mathrm{Aut}(C)$ consists of inner automorphisms (Skolem-Noether Theorem) and so is isomorphic to $PC^{\times}$. When $C\cong M(2,F)$, then $\mathrm{Aut}(C)\cong \mathrm{PGL}_{2}(F)$.

When $\mathrm{dim}_{F}(C)=8$, then $C$ is isomorphic to $\mathbb{O}$ the split octonian algebra over $F$. Its automorphism group is the $F$-points of an algebraic group of type $G_{2}$.

For more information on composition algebras the reader can refer to \cite{J85, SV00}

\subsection{Jordan Algebras}\label{JordanAlg}

Next we describe a family of Jordan algebras indexed by a composition algebra $C$. For more details, see Pollack \cite[Chapter 2, Section 2]{P22}. 

Let
\begin{equation*}
\mathcal{J}=\mathcal{J}_{C}=\Big\{X=\left(\begin{smallmatrix}
c_{1}& x_{3} & \overline{x_{2}}\\
\overline{x_{3}} & c_{2} & x_{1}\\
x_{2} & \overline{x}_{1} & c_{3}
\end{smallmatrix}\right)|c_{j}\in F,\, x_{j}\in C\Big\}.
\end{equation*}
Given $A,B\in \mathcal{J}_{C}$ the Jordan multiplication is defined by 
\begin{equation*}
A* B=\frac{AB+BA}{2},
\end{equation*}
where $AB$, $BA$ denotes usual matrix multiplication. The algebra $\mathcal{J}_{C}$ is equipped with a cubic norm form 
\begin{equation*}
N_{\mathcal{J}}(X)=c_{1}c_{2}c_{3}-c_{1}n_{C}(x_{1})-c_{2}n_{C}(x_{2})-c_{3}n_{C}(x_{3})+\mathrm{Tr}(x_{1}x_{2}x_{3}).
\end{equation*} 
The norm form uniquely defines a symmetric trilinear form $(-,-,-)_{\mathcal{J}}:\mathcal{J}\times \mathcal{J}\times \mathcal{J}\rightarrow F$ normalized so that $(X,X,X)_{\mathcal{J}}=6N_{\mathcal{J}}(X)$.

We write $\mathrm{Tr}_{\mathcal{J}}:\mathcal{J}\rightarrow F$ for the map defined by $\mathrm{Tr}_{\mathcal{J}}(X)=c_{1}+c_{2}+c_{3}$. From this we define a nondegenerate pairing $\langle-,-\rangle_{\mathcal{J}}:\mathcal{J}\times\mathcal{J}\rightarrow F$ by $\langle X,X^{\prime}\rangle_{\mathcal{J}}=\mathrm{Tr}_{\mathcal{J}}(X*X^{\prime})$.

There is also a map $\phantom{X}^{\#}:\mathcal{J}\rightarrow\mathcal{J}$ defined by
\begin{equation*}
X^{\#}= 
\left(\begin{smallmatrix}
c_{2}c_{3}-n_{C}(x_{1})& \overline{x}_{2}\overline{x}_{1}-c_{3}x_{3} & x_{3}x_{1}-c_{2}\overline{x_{2}}\\
x_{1}x_{2}-c_{3}\overline{x}_{3} & c_{1}c_{3}-n_{C}(x_{2}) & \overline{x}_{3}\overline{x}_{2}-c_{1}x_{1}\\
\overline{x}_{1}\overline{x}_{3}-c_{2}x_{2} & x_{2}x_{3}-c_{1}\overline{x_{1}} & c_{1}c_{2}-n_{C}(x_{3})
\end{smallmatrix}\right).
\end{equation*}
This map can be used to define the cross product 
\begin{equation}\label{CrossProd1}
X\times Y=(X+Y)^{\#}-X^{\#}-(Y)^{\#}.
\end{equation}
Alternatively $X\times  Y \in\mathcal{J}$ is the unique element such that for all $Y\in \mathcal{J}$
\begin{equation}\label{CrossProd2}
\langle X\times Y, Z \rangle_{\mathcal{J}}= (X, Y ,Z)_{\mathcal{J}}.
\end{equation}

There is a notion of rank for elements in $\mathcal{J}$. Every element $X\in\mathcal{J}$ has rank at most $3$. If $N(X)=0$, then $X$ has rank at most $2$. If $X^{\#}=0$, then $x$ has rank at most 1. If $X=0$, then $X$ has rank $0$. (Pollack \cite[Chapter 3, Section 3]{P22})

\vskip 5pt 

We write $H_{C}$ for the group of invertible linear transformations of $\mathcal{J}=\mathcal{J}_{C}$ that scale the norm form $N_{\mathcal{J}}$ (i.e the 
group of similitudes of the cubic form), and $H_{C}^{1}$ for the subgroup preserving the norm form. 
We have a subgroup $\Aut(C) \times \GL_3(F) \rightarrow H_{C}$ where $g\in \Aut(C)$ acts naturally on entries of  elements of $\mathcal J$, while  
$h\in \GL_3(F)$  acts on $X\in \mathcal J$ by 
\[ 
\det(h) \cdot (h^{-1})^{\top} X h^{-1},
\] 
where $h^{\top}$ denotes the transpose of $h$. The similitude character of this transformation of $\mathcal J$ is $\det(h)$.  
  If $C=F$, then $H_F \cong \GL_3(F)$. 
In general, $\Aut (C) \times \GL_3(F)$ preserves the decomposition 
\[ 
\mathcal J_{C} =\mathcal J_F \oplus \mathcal J_{C^0} 
\] 
where $\mathcal J_{C^0}$ is the subspace consisting of 
\begin{equation*}
J(x)=\left(\begin{smallmatrix}
0& x_{3} & \overline{x_{2}}\\
\overline{x_{3}} & 0 & x_{1}\\
x_{2} & \overline{x}_{1} & 0
\end{smallmatrix}\right),
\end{equation*}
where $x=(x_{1},x_{2},x_{3})\in (C^{0})^3$.  
The following proposition in essence restates the 
known fact that the dual of the standard $3$-dimensional representation of $\GL_3$ is isomorphic to the exterior square of the standard representation twisted by 
determinant inverse. In any case it is easy to check. 
\begin{proposition}  Let $V_3$ be the standard representation of $\GL_3(F)$. Then $\mathcal J_{C^0}\cong C^0\otimes V_3$. Explicitly,  
$(g,h)\in \Aut(C) \times \GL_3(F)$ acts on $J(x)$ by $J(gxh^{\top})$. 

\end{proposition} 


\subsection{Construction of exceptional Lie algebras}\label{DualPairs}

Let $\mathfrak{h}_{C}$ be the Lie algebra of $H_{C}^{1}$. We define vector spaces
\begin{align*}
\mathfrak{g}_{0,C}=&\mathfrak{sl}(3,F)\oplus \mathfrak{h}_{C},\\
\mathfrak{g}_{1,C}=&V_{3}\otimes \mathcal{J}_{C},\\
\mathfrak{g}_{-1,C}=&V_{3}^{*}\otimes \mathcal{J}_{C}^{*},
\end{align*}
where $V_3$ is the standard representation of $\mathfrak{sl}_3$ and $V_3^*$ the dual of $V_3$.  We identify $\mathcal{J}_{C}^{*}$ with $\mathcal{J}_{C}$ using the 
trace form. Consider the vector space
\begin{equation*}
\mathfrak{g}_{C}=\mathfrak{g}_{0,C}\oplus\mathfrak{g}_{1,C}\oplus\mathfrak{g}_{-1,C}.
\end{equation*}
The space $\mathfrak{g}_{C}$ can be given the structure of a Lie algebra that extends the Lie algebra structure on $\mathfrak{g}_{0,C}$ and the natural action of $\mathfrak{g}_{,C}$ on $\mathfrak{g}_{1,C}\oplus \mathfrak{g}_{-1,C}$. (See \cite{R97}, Section 1.3.) We write $\langle-,-\rangle_{C}$ for the Killing form of $\mathfrak{g}_{C}$. The Lie algebra $\mathfrak{g}_{F}$ is the split simple Lie algebra of type $F_{4}$. The Lie algebra $\mathfrak{g}_{C}$ is a simple Lie algebra of type $E_{n}$, where $n=6,7,8$ when $\mathrm{dim}_{F}(C)=2,4,8$, respectively.

Let $Y\in\mathfrak{sl}(3,F)$ and $X\in\mathcal{J}_{C}$, then $X\mapsto YX+XY^{\top}$ defines a Lie algebra action of $\mathfrak{sl}(3,F)$ on $\mathcal{J}_{C}$ extending the analogous action of $\mathfrak{sl}(3,F)$ on $\mathcal{J}_{F}$. This action induces an inclusion of Lie algebras $\mathfrak{h}_{F}\cong \mathfrak{sl}(3,F)\hookrightarrow \mathfrak{h}_{C}$, which induces an inclusion of Lie algebras $\mathfrak{g}_{F}\hookrightarrow \mathfrak{g}_{C}$.

Next we describe a Heisenberg parabolic subalgebra in $\mathfrak{g}_{C}$. Let $\mathfrak{t}$ be the subalgebra of diagonal matrices in $\mathfrak{sl}(3,F)\subset \mathfrak{g}_{0,C}$. The adjoint action of $\mathfrak{t}$ on $\mathfrak{g}_{C}$ provides a decomposition 
\begin{equation*}
\mathfrak{g}_{C}=\bigoplus_{\gamma\in \mathfrak{t}^{*}}\mathfrak{g}_{\gamma},
\end{equation*}
where $\mathfrak{g}_{\gamma}=\{X\in \mathfrak{g}_{C}|[h,X]=\gamma(h)X\text{ for all }h\in\mathfrak{t}\}$. The weights $\gamma\neq 0$ such that $\mathfrak{g}_{\gamma}\neq 0$ form a relative root system $\underline{\Phi}$ of type $G_{2}$ (\cite[Sections 9.2, 10.8]{GS05}). 
Note that the long relative root spaces are all isomorphic to $F$ and sit in $\mathfrak{sl}(3,F)$ while the short relative root spaces are isomorphic to $\mathcal{J}_{C}$ or $\mathcal{J}_{C}^{*}$.
Finally, observe that $\mathfrak g_0= \mathfrak{t} \oplus \mathfrak{h}_C$.  

\vskip 5pt 

We let $\{\alpha,\beta\}$ be a set of simple roots in the $G_{2}$ relative root system so that $\alpha$ is long, $\beta$ is short. 
Then the maximal root  $\alpha_{\max}=2\alpha+3\beta$ is a long root. Thus, without loss of generality, we can assume that 
$h_{\alpha_{\max}}=\mathrm{diag}(1,0,-1)\in\mathfrak{sl}(3,F)$.

\includegraphics[scale=1]{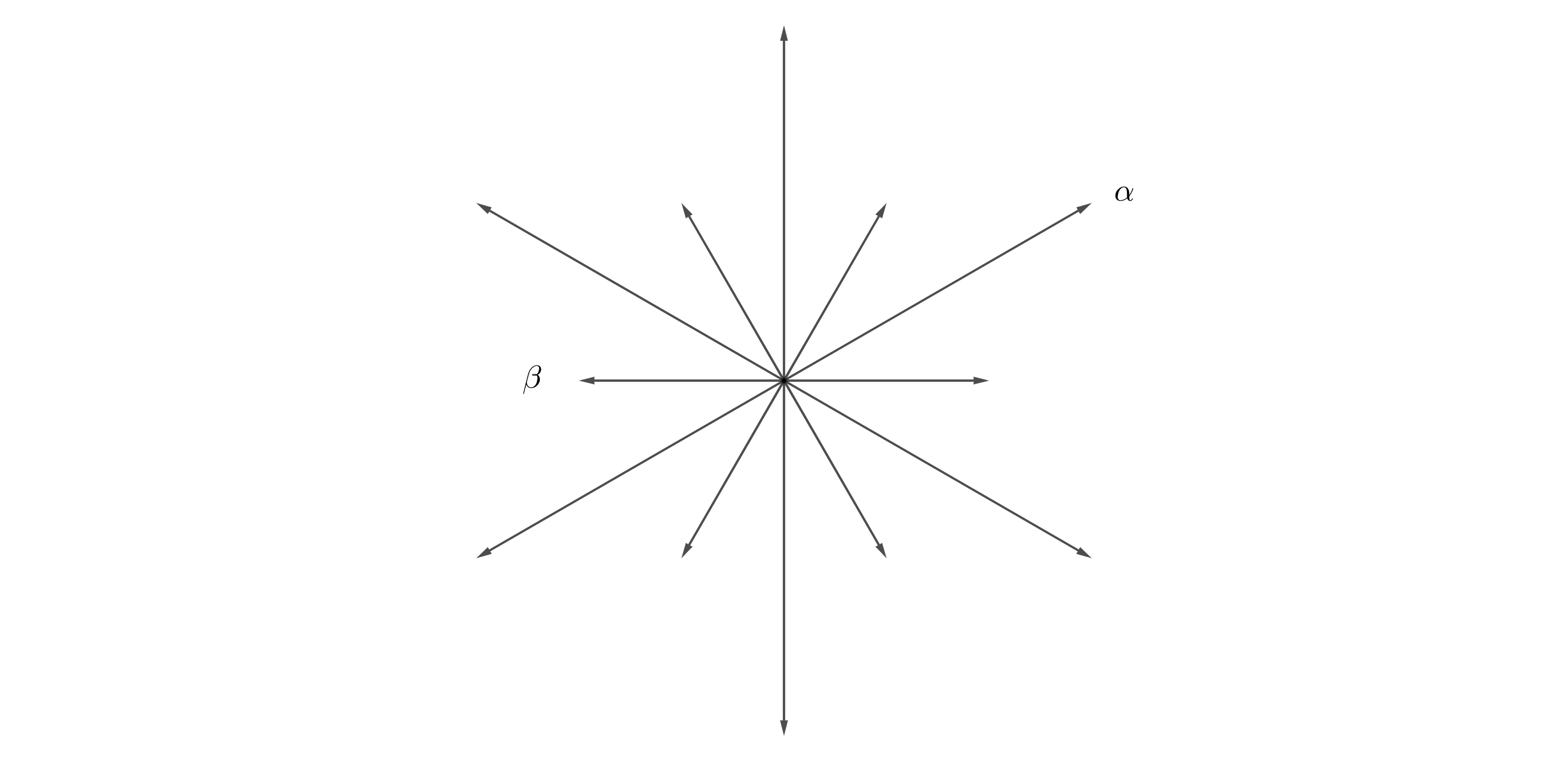}

The element $h_{\alpha_{\max}}$ defines a $\mathbb{Z}$-grading on $\mathfrak{g}_{C}$ supported on $\{0,\pm1,\pm2\}$. For $j\in \mathbb{Z}$, let
\begin{equation*}
\mathfrak{g}_{C}(j)=\{x\in \mathfrak{g}_{C}~| ~[h_{\alpha_{\max}},x]=jx\}.
\end{equation*}

Let $\mathfrak{p}=\oplus_{j\geq 0}\mathfrak{g}_{C}(j)$. Then $\mathfrak{p}$ is a Heisenberg parabolic subalgebra with Levi subalgebra 
\[ 
\mathfrak{m}=\mathfrak{g}_{C}(0)= \mathfrak{t}\oplus\mathfrak{h}_{C}\oplus\mathfrak{g}_{\beta}\oplus\mathfrak{g}_{-\beta} 
\]
 and nilpotent radical $\mathfrak{n}=\oplus_{j>0}\mathfrak{g}_{C}(j)$ with one dimensional center 
\[ 
\mathfrak{z}=\mathfrak{g}_{C}(2)= \mathfrak{g}_{\alpha_{\max}}. 
\]

Let $\mathcal{G}_{C}=\mathrm{Aut}(\frak{g}_{C})$. If $C\neq F$ then 
the connected component of $\mathcal{G}_{C}$ is an adjoint group of type $E_{n}$. The group $\mathcal G_F$ is $F_4$.
We omit the subscript $C$ when no confusion can arise. 
Then the maximal parabolic subalgebra $\mathfrak p$ corresponds to a maximal parabolic subgroup $\mathcal P=\mathcal M \mathcal N$ in $\mathcal G$. 
Let $\mathcal Z$ be the center of $\mathcal N$. Then $\mathcal M$ acts on $\mathcal N/\mathcal Z \cong \mathfrak{n}/\mathfrak{z}$. 
The space $\mathfrak{n}/\mathfrak{z}$ admits a symplectic and a quartic form, and $\mathcal M$ acts as a group of similitudes of these two forms.

\subsection{A symplectic space}\label{HeisenbergSpace} 
Using Pollack \cite[Chapter 3]{P22} we give an explicit construction of the  reductive group $\mathcal M$ and its representation on the symplectic space 
$\mathfrak{n}/\mathfrak{z}$. In terms of the restricted root system, we have 
\[ 
\mathfrak{n}/\mathfrak{z}\cong \mathfrak g_{\alpha} \oplus \mathfrak g_{\alpha+\beta} \oplus \mathfrak g_{\alpha+2\beta} \oplus \mathfrak g_{\alpha+3\beta}.
\] 
We identify $\mathcal J$ and $\mathcal J^*$ using the trace form, so $\mathfrak g_{\gamma}\cong \mathcal J$ for any short root $\gamma$. Thus we can identify 
 $\mathfrak{n}/\mathfrak{z}$ with 
\begin{equation*}
\mathbb{W}=\mathbb{W}_{C}=F\oplus \mathcal{J}\oplus\mathcal{J}\oplus F.
\end{equation*}
So any element  $w\in \mathbb{W}$ is a quadruple $w=(a,b,c,d)$, where $a,d\in F$ and $b,c\in \mathcal{J}$. The space  $\mathbb W$ comes with a symplectic form 
\begin{equation*}
\langle(a,b,c,d),(a^{\prime},b^{\prime},c^{\prime},d^{\prime})\rangle_{\mathbb{W}}=ad^{\prime}-\mathrm{Tr}(b*c^{\prime})+\mathrm{Tr}(c*b^{\prime})-da^{\prime}, 
\end{equation*}
and a quartic form
\begin{equation*}
q(a,b,c,d)=(ad-\mathrm{Tr}(b*c))^{2}+4aN(c)+4dN(b)-4\mathrm{Tr}(b^{\#}*c^{\#}).
\end{equation*}
Let $(-,-,-,-)_{\mathbb{W}}$ be the unique symmetric $4$-linear form on $\mathbb{W}$ such that $(v,v,v,v)=2q(v)$.
 Then $\mathcal M$ is isomorphic to the group of similitudes
\begin{equation*}
M_{C}=\{(g,\nu)\in \mathrm{GL}(\mathbb{W})\times \mathrm{GL}_1(F)|\langle gv,gv^{\prime}\rangle=\nu\langle v,v^{\prime}\rangle,\,q(gv)=\nu^{2}q(v)\text{ for all }v,v^{\prime}\in \mathbb{W}\}.
\end{equation*}
We write $M_{C}^{1}$ for the subgroup of elements where the similitude factor $\nu$ is equal to $1$. 

\vskip 5pt 

We highlight a few subgroups of $M^1_{C}$.
If $h\in H_{C}$ with similitude factor $\lambda$, then the map $(a,b,c,d)\mapsto (\lambda a,hb,\tilde{h}c,\lambda^{-1}d)$, where the action of $\tilde{h}$ on $\mathcal{J}$ is defined through the identification of $\mathcal{J}$ with $\mathcal{J}^{*}$ via the trace pairing, defines an element of $M^1_{C}$. We abuse notation and let $H_C$ denote this subgroup.

For $x\in \mathcal{J}$ let $n(x)$ be the map defined by
\begin{equation}\label{JordanUni}
n(x)(a,b,c,d)=(a,b+ax,c+b\times x+ax^{\#},d+\mathrm{Tr}(c*x)+\mathrm{Tr}(b*x^{\#})+aN(x)).
\end{equation}
The map $n(x)\in M_{\mathcal{J}}$ and has similitude factor equal to $1$. The group generated by these elements is isomorphic to 
the unipotent group $\exp(\mathfrak g_{\beta}) \subset \mathcal M$. 

Similarly, for $x\in \mathcal{J}$ let $\overline{n}(x)$ be the map defined by
\begin{equation*}
\overline{n}(x)(a,b,c,d)=(a+\mathrm{Tr}(b*x)+\mathrm{Tr}(c*x^{\#})+dN(x),b+c\times x+dx^{\#},c+dx,d).
\end{equation*}
The map $\overline{n}(x)\in M_{C}$ and has similitude factor equal to $1$. 
The group generated by these elements is isomorphic to 
the unipotent group $\exp(\mathfrak g_{-\beta}) \subset \mathcal M$. 

The two abelian groups generated by $n(x)$ and $\overline n(x)$, respectively, are unipotent radicals of two opposite maximal parabolic subgroups in $M_{C}^1$ 
with the Levi factor $H_{C}$.  These two parabolic groups are conjugate by 
\begin{equation}\label{WJInv}
(a,b,c,d)\mapsto (-d,c,-b,a).
\end{equation} 

\vskip 5pt 
 If $\lambda\in \mathrm{GL}_{1}(F)$, then 
 \[ 
 s_{\lambda}:(a,b,c,d)\mapsto (\lambda^{2},\lambda b,c,\lambda^{-1}d)
 \] 
 and 
 \[ 
 s^*_{\lambda}:(a,b,c,d)\mapsto (\lambda^{-1}, b,\lambda c,\lambda^{2}d)
 \] 
 are two elements of $M_C$ with the similitude factor $\nu=\lambda$. Thus $M_C$ is generated by $M^1_C$ and any of the 
 two one-parameter groups $s$ or $s^{\ast}$. We have written down both of these two groups for the sake of symmetry but also because they generate a 
 two dimensional torus whose Lie algebra is $\mathfrak t \subset \mathfrak m$.   
 
 \vskip 5pt 

Observe that $\Aut(C) \subset M_C$ where $\Aut(C)$ acts on the coordinates of $\mathbb W_C$. 
The centralizer of $\Aut(C)$ in $M_C$ is $M_F$, the Levi of the Heisenberg maximal parabolic of $F_4$. 
This group is isomorphic to $\mathrm{GSp}_6(F)$, as one can see from root data, for example.  We shall fix an isomorphism $M_F\cong \mathrm{GSp}_6(F)$
as follows.  Recall that $\mathbb W_C$ is a symplectic space. Under the action of $\Aut(C)$ it decomposes as 
\[ 
\mathbb W_C=\mathbb W_F \oplus (\mathcal J_{C^0} \oplus \mathcal J_{C^0}).
\] 
If $V_6$ is a 6-dimensional symplectic space then $C^0\otimes V_6$ is a symplectic space obtaining by tensoring the quadratic space $C^0$ and the symplectic space $V_6$. 
We pick $V_6$ so that 
 \begin{equation}\label{SympId}
 \mathcal J_{C^0} \oplus \mathcal J_{C^0}\cong C^0\otimes V_6, 
 \end{equation} 
 given by $(J(x),J(y))\mapsto (x_1,x_2,x_3,y_1,y_2,y_3)$, is an isomorphism of symplectic spaces. 
  Since $M_F$ commutes with $\Aut(C)$, and $\Aut(C)$ acts on $C^0$ irreducibly, $M_F$ must act on $V_6$, giving an identification 
 with $\mathrm{GSp}_6(F)$. Let $\mathrm{sim}$ denote the usual similitude character of $\mathrm{GSp}_6(F)$. 
 Observe that the similitude character of $M_{C}$ restricts to $\mathrm{sim}$  under the identification.  
 

\subsection{Orbits}\label{Orbits}   
We now describe orbits of $M_{\mathcal J}$ acting on $\mathbb W=\mathbb W_{C}$.  
Given $v=(a,b,c,d)\in \mathbb{W}_{\mathcal{J}}$ define $v^{\flat}=(a^{\flat},b^{\flat},c^{\flat},d^{\flat})$ (\cite[Proposition 1.0.3]{P22}), where
\begin{align*}
a^{\flat}=&-a(ad-\mathrm{Tr}(b*c))-2N(b);\\
b^{\flat}=&-2c\times b^{\#}+2ac^{\#}-(ad-\mathrm{Tr}(b*c))b;\\
c^{\flat}=&2b\times c^{\#}-2bd^{\#}+(ad-\mathrm{Tr}(b*c))c;\\
d^{\flat}=&d(ad-\mathrm{Tr}(b*c))+2N(c).
\end{align*}
Over the algebraic closure the orbits are classified by the rank for elements in $\mathbb{W}$, defined as follows. 
 Let $v\in \mathbb{W}$. The element $v$ has rank at most 4. If $q(v)=0$, then $v$ has rank at most 3. If $v^{\flat}=0$, then $v$ has rank at most 2. If $(v,v,w,w^{\prime})=0$ for all $w,w^{\prime}\in (v)^{\perp}$ (the orthogonal complement with respect to $\langle-,-\rangle_{\mathbb{W}}$), then $v$ has rank at most 1. If $v=0$, then $v$ has rank 0.

We need the following proposition \cite[Proposition 8.1]{GS22} and a simple corollary.

\begin{proposition} \label{WRank1} 
 A non-zero element  $(a,b,c,d)\in \mathbb{W}$ has rank 1 if and only if 
\begin{enumerate}
\item $b^{\#}-ac=0$,
\item $c^{\#}-db=0$,
\item $ad=h(b) \ast  \tilde h(c)$ for all $h\in H_{C}$, where $\tilde h$ is the dual action of $h$ on $\mathcal J^{\ast}$ identified with $\mathcal J$ using the trace form. \label{rk1Con3}
\end{enumerate}
\end{proposition}

\begin{corollary}\label{SpecialRk1}
 $\Xi=(1,b,c, d)\in \mathbb{W}$ has rank 1 if and only if $c=b^{\#}$ and $d=N(b)$. 
\end{corollary}

\begin{proof} If $\Xi$ has rank 1, then Proposition \ref{WRank1} implies $c=b^{\#}$  and $d=b *b^{\#}= N(b)$ (use $h=1$ in (\ref{rk1Con3})). 
In the opposite direction use $(b^{\#})^{\#}=N(b) \cdot b$ and $\tilde h(b^{\#})= h(b)^{\#}/\nu$ where $\nu$ is the similitude factor of $h$.  
\end{proof}


\subsection{The dual pair $\Aut(C)\times \mathcal G_F$ in $\mathcal G_C$}\label{DualPair}
Observe that $\Aut(C)$ naturally acts on $\mathcal J_C$ preserving the norm $N_{\mathcal J}$.  Thus $\mathrm{Der}(C)$, the Lie algebra of $\Aut(C)$, is a subalgebra of 
$\mathfrak h_C$. From the construction of $\frak{g}_{C}$ in Subsection \ref{DualPairs} it is evident that 
the centralizer of $\mathrm{Der}(C)$ in $\mathfrak{g}_C$ contains $\mathfrak{g}_F$, the Lie algebra of $\mathcal G_F$, the group of type 
$F_4$. This group is simply connected with trivial center. Thus the inclusion of Lie algebras lifts to an inclusion $\mathcal G_F\subseteq \mathcal G_C$. 

\begin{lemma}\label{MAX}  $\Aut(C) \times \mathcal G_F$ is a maximal subgroup in $\mathcal{G}_C$.  
\end{lemma} 
\begin{proof} 
Observe that, under the adjoint action restricted to $\mathrm{Der}(C) \oplus \mathfrak g_{F}$, 
\[ 
\mathfrak g_C= \mathrm{Der}(C) \oplus \mathfrak g_{F} \oplus C^0\otimes V_{26} 
\] 
where $C^0$ and $V_{26}$ are irreducible representations of $\mathrm{Der}(C)$ and $\mathfrak g_{F}$ respectively. The irreducibility of $C^0\otimes V_{26}$  implies that 
$\mathrm{Der}(C) \oplus \mathfrak g_{F}$  is a maximal subalgebra of $\mathfrak g$. 

First assume $\dim(C)\neq 2$.  Then $\Aut(C) \times \mathcal G_F$ is connected. 
Since $\mathrm{Der}(C) \oplus \mathfrak g_{F}$ is maximal, any proper subgroup $H$ of $\mathcal{G}_{C}$  containing $\Aut(C) \times \mathcal G_F$ must have 
$\Aut(C) \times \mathcal G_F$  as its connected component of identity. Now $\Aut(C) \times \mathcal G_F$ has no 
outer automorphisms, so if $H$ is disconnected, then there exists a semisimple finite order element $z\in \mathcal{G}_C$ centralizing $\Aut(C) \times \mathcal G_F$. 
Since $C^0\otimes V_{26}$ is irreducible, $z$ would have to act on it as $-1$. But the centralizer of the semisimple element $z$ must contain a maximal torus. Since $\Aut(C) \times \mathcal G_F$ has rank $5$ and $\mathcal{G}_{C}$ has rank 7, this is a contradiction.
 
 Second, assume $\dim C=2$. Note that in this case $\mathcal{G}_{C}$ is disconnected and the component group is generated by the outerautomorphism of the $E_{6}$  Dynkin diagram, which has order $2$. As above we can see that $\mathcal{G}_{F}$ is maximal in $\mathcal{G}_{C}^{\circ}$, the connected component of the identity. Since the generator of $\Aut(C)=\mu_2$ is the outerautomorphism in $\mathcal{G}_{C}$, the conclusion of the lemma holds in this case, too. 
\end{proof}

We are now in a position to understand  rational $\mathcal G_C$-conjugacy classes of $\Aut(C) \times \mathcal G_F\subset \mathcal{G}_C$ over $F$. 
Over $\bar F$, there is one conjugacy class, that is, any subgroup of $\mathcal G_C$ isomorphic to $\Aut(C) \times\mathcal{G}_F$  
is conjugate to $\Aut(C) \times \mathcal{G}_F$,  this is due to Dynkin. By Lemma \ref{MAX}, the normalizer of  $\Aut(C) \times \mathcal G_F$ 
in $\mathcal{G}_C$ is  $\Aut(C) \times \mathcal G_F$, hence conjugacy classes over $F$ correspond to the kernel of the map of pointed sets 
\[ 
H^1(F, \Aut(C)) \times H^1(F, \mathcal G_F) \rightarrow H^1(F, \mathcal{G}_C). 
\] 
If $F$ is $p$-adic, then $H^1(F, \mathcal G_F)$ is trivial, hence we are reduced to the map $H^1(F, \Aut(C)) \rightarrow H^1(F, \mathcal{G}_C)$. 
This maps sends a rational form $C'$ of $C$ (i.e. an element of $H^1(F, \Aut(C))$) to $\mathcal G_{C'}$. This map is clearly injective, hence 
we have only one conjugacy class. Furthermore, injectivity implies that $\Aut(C) \times \mathcal G_F$ can be a subgroup of $\mathcal G_{C'}$ only when 
$C\cong C'$.  

In this paper we use two different constructions of the dual pair $\mathrm{Aut}(C)\times \mathcal{G}_{F}\subseteq \mathcal{G}_{C}$ to investigate the theta correspondence. From the above these two subgroups are conjugate. Thus the results obtained using each construction are compatible.

\subsection{Degenerate Principal series on $F_{4}$}

In this section, we collect the results of Choi-Jantzen \cite{CJ10} that describe the structure of degenerate principal series on $F_{4}$. Henceforth we write $G=\mathcal G_F$ 
for the unique group of type $F_4$ over $F$. 

We note that in this paper we use the Bourbaki labeling of simple roots of $F_{4}$ \cite[Plate VIII]{B02}, but Choi-Jantzen \cite{CJ10} use the reverse order. i.e. the two labelings are related by the permutation $(14)(23)$, written in disjoint cycle notation.

\begin{proposition}[Theorems 3.1 and 6.1 \cite{CJ10}]\label{CJDPS}
Let $\chi$ be a character of $F^{\times}$ such that $\chi=|-|^{s}\chi_{0}$, where $\chi_{0}$ is a unitary character and $s\in \mathbb{C}$. Let $Q$ be a maximal parabolic subgroup associated with the simple root $\alpha_{4}$ and let $\omega_{4}$ be the fundamental weight that pairs nontrivially with $\alpha_{4}^{\vee}$.
\begin{enumerate}
\item $i_{Q}^{G}(\chi\circ\omega_{4})$ is reducible if and only if $s= \pm\frac{11}{2}, \pm\frac{5}{2},\pm\frac{1}{2}$ with $\chi_{0}$ trivial, or $s=\pm\frac{1}{2}$ with $\chi_{0}$ of order $2$.
\item If $s= \pm\frac{11}{2}, -\frac{5}{2},\pm\frac{1}{2}$ and $\chi_{0}$ has order dividing $2$, then $i_{Q}^{G}(\chi\circ\omega_{4})$ has a unique irreducible quotient.
\item If $s=\frac{5}{2}$ and $\chi_{0}$ is trivial, then $i_{Q}^{G}(\chi\circ\omega_{4})$ has a maximal semisimple quotient of the form $\sigma^{+}\oplus\sigma^{-}$, where $\sigma^{+}$ and $\sigma^{-}$ are distinct irreducible representations of $G$. 
\end{enumerate}
\end{proposition}

\subsection{$\mathrm{PGL}_{2}(F)$}

In this subsection, we recall basic facts about $\mathscr G= \mathrm{PGL}_{2}(F)$.

\begin{lemma}
The characters of $\mathscr{G}$ are in bijection with the quadratic characters of $F^{\times}$.  
\end{lemma}

Let $\mathscr{B}=\mathscr T \mathscr{U}$ be a Borel subgroup of $\mathscr G$ and let $\overline{\mathscr{B}}=\mathscr T \overline{\mathscr{U}}$ be its opposite.  We fix an identification $\mathscr T\cong F^{\times}$ so that 
$\mathscr T$ acts on $\mathscr{\overline{U}}$ by multiplication. In particular, the modular character is $\delta_{\overline{\mathscr{B}}}(t)=|t|$.  
Let $\chi$ be a character of $\mathscr{T}$ and define $i_{\overline{\mathscr{B}}}^{\mathscr{G}}(\chi)=\mathrm{Ind}^{\mathscr{G}}_{\overline{\mathscr{B}}}(\delta_{\overline{\mathscr{B}}}^{1/2}\cdot\chi)$.  Let $\mathrm{St}$ denote the Steinberg representation of $\mathscr G$. 
We have the following well known result: 

\begin{lemma}\label{PGL2PS} Let $\chi$ be a character of $F^{\times}$ such that $\chi=|-|^{s}\chi_{0}$, where $\chi_{0}$ is a unitary character and $s\in \mathbb{C}$.  
The representation $i_{\overline{\mathscr{B}}}^{\mathscr{G}}(\chi)$ is irreducible unless $s=\pm 1/2$ and $\chi_0$ is a quadratic character. 
\begin{enumerate}
\item If $s=1/2$ and $\chi_0$ is a quadratic character. Then $i_{\overline{\mathscr{B}}}^{\mathscr{G}}(\chi)$ has length 2.  The unique irreducible 
quotient is the one dimensional representation obtained by inflating $\chi_0$  to $\mathscr G$.  The unique irreducible submodule is 
 $\mathrm{St}\otimes \chi_0$.  
\item If $s=-1/2$ and $\chi_0$ is a quadratic character. Then $i_{\overline{\mathscr{B}}}^{\mathscr{G}}(\chi)$ has length 2.  The unique irreducible 
quotient is $\mathrm{St}\otimes \chi_0$. The unique irreducible submodule the one dimensional representation obtained by inflating $\chi_0$  to $\mathscr G$. 
\end{enumerate}
\end{lemma}

\subsection{Theta lifting preliminaries}

In this subsection we establish preliminaries to discuss a theta lift for $\mathrm{Aut}(C)\times F_{4}\subset E_{7}$. 

Recall that associated to a quaternion algebra $C$ we have the following groups: $\mathcal{G}$ is the $F$-points of a connected semisimple adjoint group of type $E_{7}$; $G$ is the $F$-points of a connected semisimple group of type $F_{4}$; $\mathscr{G}$ is the $F$-points of the automorphism group of $C$. Let $(\Pi,\mathcal{V})=(\Pi_{\mathrm{min}},\mathcal{V}_{\mathrm{min}})$ be the minimal representation of $\mathcal{G}$ (Savin \cite[Theorem 2.2]{S94}).

To begin we define the big theta lift. Let $\tau$ be an irreducible smooth $\mathscr{G}$-representation. The maximal $\tau$-isotypic quotient of $\mathcal{V}$ is naturally a $\mathscr{G}\times G$-representation $\mathcal{V}_{\tau}$. Furthermore, $\mathcal{V}_{\tau}$ admits a factorization $\mathcal{V}_{\tau}\cong \tau\otimes \Theta(\tau)$, where $\Theta(\tau)$ is a $G$-representation. We call $\Theta(\tau)$ the big theta lift of $\tau$ with respect to the restriction of $\mathcal{V}$ to $\mathscr{G}\times G$. (For details, see \cite[Chapter 2, Section 3]{MVW87}.) Let $\theta(\tau)$ be the maximal semisimple quotient (cosocle) of $\Theta(\tau)$.

The main objective of this work is to investigate $\Theta(\tau)$ and $\theta(\tau)$. The following simple lemma is important for our analysis. If $V$ is a vector space, we write $V^{*}$ for its linear dual.

\begin{lemma}\label{ThetaDual}
Let $\tau\in \mathrm{Irr}(\mathscr{G})$. Let $U\subset G$ be a unipotent subgroup with a character $\Psi$ and let $H=\mathrm{Stab}_{G}(U,\Psi)$. There is an $H$-module isomorphism
\begin{equation*}
(\Theta(\tau)_{(U,\Psi)})^{*}\cong \mathrm{Hom}_{\mathscr{G}}(\mathcal{V}_{(U,\Psi)},\tau),
\end{equation*}
where the $H$ acts on $\mathrm{Hom}_{\mathscr{G}}(\mathcal{V}_{(U,\Psi)},\tau)$ by $(h\cdot f)(v)=f(h^{-1}\cdot v)$. 

In particular, there is an isomorphism of $G$-modules
\begin{equation*}
\Theta(\tau)^{*}\cong \mathrm{Hom}_{\mathscr{G}}(\mathcal{V},\tau).
\end{equation*}
\end{lemma}

\begin{proof} Let $\phi:\mathcal{V}\twoheadrightarrow \tau\otimes \Theta(\tau)$ be a surjection and $\mathrm{pr}:\mathcal{V}\rightarrow \mathcal{V}_{(U,\Psi)}$ the canonical quotient. 

First we show there is an $H$-module isomorphism 
\begin{equation*}
\frak{F}:\mathrm{Hom}_{\mathscr{G}}(\tau\otimes \Theta(\tau)_{(U,\Psi)},\tau)\rightarrow \mathrm{Hom}_{\mathscr{G}}(\mathcal{V}_{(U,\Psi)},\tau),
\end{equation*}
defined by $f\mapsto f\circ\phi_{(U,\Psi)}$. The map $\frak{F}$ is injective because $\phi_{(U,\Psi)}$ is surjective.

Now let $g\in\mathrm{Hom}_{\mathscr{G}}(\mathcal{V}_{(U,\Psi)},\tau)$. Then the map $g\circ\mathrm{pr}\in \mathrm{Hom}_{\mathscr{G}}(\mathcal{V},\tau)$ factors through $\tau\otimes \Theta(\tau)$ via $\phi$ which then implies that $g$ factors through $\tau\otimes \Theta(\tau)_{(U,\Psi)}$. Thus $g=f\circ \phi_{(U,\Psi)}$ for some $f\in \mathrm{Hom}_{\mathscr{G}}(\tau\otimes \Theta(\tau)_{(U,\Psi)},\tau)$. So, $\frak{F}$ is surjective.

Finally, we claim that the tensor-Hom adjunction and Schur's lemma give a canonical isomorphism $\mathrm{Hom}_{\mathscr{G}}(\tau\otimes \Theta(\tau)_{(U,\Psi)},\tau)\cong (\Theta(\tau)_{(U,\Psi)})^{*}$, which completes the result. 

The tensor-Hom adjunction gives a canonical $H$-module isomorphism 
\begin{equation*}
\mathrm{Hom}_{\mathscr{G}}(\tau\otimes \Theta(\tau)_{(U,\Psi)},\tau)\rightarrow \mathrm{Hom}(\Theta(\tau)_{(U,\Psi)},\mathrm{Hom}_{\mathscr{G}}(\tau,\tau)),
\end{equation*}
defined by $F(a,b)\mapsto (b \mapsto (a\mapsto F(a,b)))$.  By Schur's lemma, there is a canonical isomorphism $\mathrm{Hom}_{\mathscr{G}}(\tau,\tau)\cong \mathbb{C}$ so that $id_{\tau}\mapsto 1$.\end{proof}

\section{Jacquet Modules I}\label{JMI}

Let $\mathcal{G}=\mathcal{G}_C$ and $G=\mathcal G_F$. Let  $\Aut(C) \times G$  be the dual pair described in subsection \ref{DualPair} 
and let  $\mathcal{P}=\mathcal{M}\mathcal{N}$ be the Heisenberg maximal parabolic subgroup  in $\mathcal G$. 
Since $\Aut(C) \subset \mathcal M$, the centralizer of $\Aut(C)$ in $\mathcal P$ is the Heisenberg maximal parabolic $P=MN$ in $G$. 
We write $\overline{\mathcal{P}}=\mathcal{M}\overline{\mathcal{N}}$ and $\overline{P}=M\overline{N}$ for the parabolic subgroups opposite to $\mathcal{P}$ and $P$, respectively.

Our study of the theta lifting is based on two tools. The first are 
functors of twisted co-invariants of the minimal representation $\mathcal V$ of $\mathcal{G}$. The second is the Fourier-Jacobi functor. 
To compute these functors, we use a $\overline{\mathcal{P}}$-filtration of $\mathcal V$ due to Magaard-Savin \cite[Theorem 6.1]{MS97}, which we now recall. 
In the following, $\Omega$ is the minimal nontrivial $\mathcal{M}$-orbit in $\mathcal{N}/Z$. Under the isomorphism $\mathcal{N}/Z\cong \mathbb W_C$, 
$\Omega$ is the set of rank 1 elements in $\mathbb W_C$. 

\begin{theorem}\label{HeisJacThm} 
Let $(\Pi,\mathcal{V})$ be the minimal representation of $\mathcal{G}$ and let $\overline{\mathcal{P}}=\mathcal{M}\overline{\mathcal{N}}$ be the Heisenberg parabolic subgroup of $\mathcal{G}$ opposite to $\mathcal{P}$. Let $\overline{Z}$ be the center of $\overline{\mathcal{N}}$. Then $\mathcal{V}$ has a $\overline{\mathcal{P}}$-filtration given by the exact sequence

\begin{equation}\label{HeisJac}
0\rightarrow C_{c}^{\infty}(\Omega)\rightarrow \mathcal{V}_{\overline{Z}}\rightarrow \mathcal{V}_{\overline{\mathcal{N}}}\rightarrow 0.
\end{equation}
Furthermore, the action of $\overline{\mathcal{P}}$ is described as follows:
\begin{enumerate}
\item Let $m\overline{n}\in \mathcal{M}\overline{\mathcal{N}}$ and $f\in C_{c}^{\infty}(\Omega)$. Then
\begin{align*}
[\Pi(\overline{n})f](x)=&\psi(\langle x,\overline{n}\rangle)f(x);\\
[\Pi(m)f](x)=& \chi_C(m) |\mathrm{det}(m)|^{s/d}f(m^{-1}\cdot x).
\end{align*}
\item $\mathcal{V}_{\overline{\mathcal{N}}}\cong [\mathcal{V}(\mathcal{M})\otimes|\mathrm{det}|^{t/d}]\oplus \chi_C|\mathrm{det}|^{s/d}$,
where $\mathcal{V}(\mathcal{M})$ is the minimal representation of $\mathcal{M}$ (center acting trivially). 
\end{enumerate}
Here $\mathrm{det}$ is the determinant of the representation of $\mathcal{M}$ acting on $\overline{\mathcal{N}}/\overline{Z}$; 
$\chi_C$ is a quadratic character, trivial unless $C$ is a quadratic field, and then corresponding to $C$ by the local class field theory;
$d$ is the dimension of $\mathcal{N}/Z$. The values of $s$, $t$, and $d$ are given in the following table.
\begin{center}
\begin{tabular}{|c||c|c|c|}\hline
$\mathcal{G}$&$s$&$t$&$d$\\ \hline\hline
$E_{6}$&$4$&$2$&$20$\\ \hline
$E_{7}$&$6$&$3$&$32$\\ \hline
$E_{8}$&$10$&$6$&$56$\\ \hline
\end{tabular} 
\end{center}
\end{theorem}

\vskip 5pt 
\noindent 
\textbf{Remark:} In Magaard-Savin \cite{MS97}, the groups are split. Nevertheless, their proof still applies to $\mathcal{G}=\mathcal{G}_{C}$, where $C$ is a split or non-split composition algebra over $F$.  The quadratic twist by $\chi_C$ was observed in \cite{GS22}. 

\vskip 5pt

Given $\tau\in \mathrm{Irr}(\mathscr{G})$ there is a surjective map $\mathcal{V}\twoheadrightarrow \tau\otimes \Theta(\tau)$. 
To study $\Theta(\tau)$ we apply the functor of $(\overline N,\Psi)$-coinvariants to sequence (\ref{HeisJac}), where $\Psi$ is a character of $\overline{N}$.
 
Let $M_{\Psi}=\mathrm{Stab}_{M}(\Psi)$. Let $\mathcal{N}(\Psi)=\{n\in \mathcal{N}|\psi(\langle n,\overline{n}\rangle)=\Psi(\overline{n})\text{ for all }\overline{n}\in \overline{N}\}$. Let $\Omega_{\Psi}=\Omega\cap \mathcal{N}(\Psi)$.

\begin{lemma}\label{HeisMinJacCalc}
The restriction map $C^{\infty}_{c}(\Omega)\rightarrow C^{\infty}_{c}(\Omega_{\Psi})$ induces a $\mathscr{G}\times M_{\Psi}$-module isomorphism 
\begin{equation*}
C^{\infty}_{c}(\Omega)_{(\overline{N},\Psi)}\cong C^{\infty}_{c}(\Omega_{\Psi}).
\end{equation*}
\end{lemma}

\begin{proof} The proof is the same as Magaard-Savin \cite{MS97}, Lemma 2.2.\end{proof}

To describe $C^{\infty}_{c}(\Omega_{\Psi})$ as a $\mathscr{G}\times M_{\Psi}$-module we need an explicit description of $\Omega_{\Psi}$, which we take up in the next subsection. When $\Psi$ is nontrivial, this gives a complete description of $\mathcal{V}_{(\overline{N},\Psi)}$, as we see in the next lemma.

\begin{lemma}\label{HeisJacTwist}
Suppose that $\Psi$ is nontrivial. By applying $(\overline{N},\Psi)$ coinvariants to the exact sequence (\ref{HeisJac}) we get a $\mathscr{G}\times M_{\Psi}$-module isomorphism
\begin{equation*}
C_{c}^{\infty}(\Omega)_{(\overline{N},\Psi)}\cong \mathcal{V}_{(\overline{N},\Psi)}.
\end{equation*}
\end{lemma}

\begin{proof} Since the functor $(-)_{(\overline{N},\Psi)}$ is exact and $(\mathcal{V}_{\overline{\mathcal{N}}})_{(\overline{N},\Psi)}=0$ the result follows.\end{proof}

\subsection{Fiber calculation}\label{JacMod1Calc}

The main objective of this section is to compute $C_{c}^{\infty}(\Omega)_{(\overline{N},\Psi)}$, where $\Psi$ is of rank 3 or rank 0. This is accomplished in Proposition \ref{Rk3Coin} for rank 3, and Proposition \ref{Rk0Coin} for rank 0. The rank of $\Psi$ will be defined in terms of the rank of elements of $\mathbb{W}_{F}$ (Subsection \ref{HeisenbergSpace}). We explain this after setting up some notation.

The map $\mathcal{N}/Z\rightarrow \mathrm{Hom}(\overline{\mathcal{N}}/\overline{Z},\mathbb{C}^{\times})$ defined by $n\mapsto \psi(\langle n, - \rangle)$ defines an isomorphism of $\mathcal{N}/Z$ with the Pontryagin dual of $\overline{\mathcal{N}}/\overline{Z}$. Similarly, by restriction this map defines an isomorphism between $N/Z$ and the Pontryagin dual of $\overline{N}/\overline{Z}$.

We identify $\mathcal{N}/Z$  with $\mathbb{W}_{C}$  and $\mathcal M$ with $M_C$ so that 
 the adjoint action of $\mathcal{M}$ on $\mathcal{N}/Z$ corresponds to the action of $M_C$ on $\mathbb{W}_{C}$. 
This also fixes an identification of  $N/Z$ with  $\mathbb{W}_{F}= \mathbb{W}_{C}^{\Aut(C)}$ and of $M\subset \mathcal M$ with $M_F \subset M_C$.  
 Thus we can view the character $\Psi$ as an element of $\mathbb{W}_{F}$. We define the rank of $\Psi$ to be the rank of its associated element in $\mathbb{W}_{F}$. 

Now we reinterpret the set $\mathcal{N}(\Psi)$ as the fiber of a map $\bold{F}$, defined below. 

Let $\bold{f}:\mathcal{J}_{C}\rightarrow \mathcal{J}_{F}$ be the map defined by 
\begin{equation*}
\left(\begin{smallmatrix}
a & x & \overline{z}\\
\overline{x} & b & y\\
z & \overline{y} & c
\end{smallmatrix}\right)
\mapsto
\left(\begin{smallmatrix}
a & \frac{\mathrm{Tr}(x)}{2} & \frac{\mathrm{Tr}(z)}{2}\\
\frac{\mathrm{Tr}(x)}{2} & b & \frac{\mathrm{Tr}(y)}{2}\\
\frac{\mathrm{Tr}(z)}{2} & \frac{\mathrm{Tr}(y)}{2} & c
\end{smallmatrix}\right).
\end{equation*}

Let $\bold{F}:\mathbb{W}_{C}\rightarrow \mathbb{W}_{F}$ be defined by
\begin{equation*}
(a,b,c,d)\mapsto (a,\bold{f}(b),\bold{f}(c),d).
\end{equation*}
With the identifications above, the natural restriction map from the Pontryagin dual of $\overline{\mathcal{N}}/\overline{Z}$ to the Pontryagin dual of $\overline{N}/\overline{Z}$ is realized as the map $\bold{F}:\mathbb{W}_{C}\rightarrow \mathbb{W}_{F}$. Viewing $\Psi$ as an element of $\mathbb{W}_{F}$, we have $\mathcal{N}(\Psi)=\bold{F}^{-1}(\Psi)$. Thus our next objective is to describe the intersection of the fibers of $\bold{F}$ with $\Omega$.

\begin{proposition}\label{NonemptyFiber} Let $C$ be any composition algebra. 
Let $\xi=(1,0,c,d)\in \mathbb{W}_{F}$. The set $\bold{F}^{-1}(\xi)\cap \Omega$  consists of 
\[ 
(1, J(x), J(x)^{\#} , N_{\mathcal J}(J(x))) 
\]
for all $x=(x_{1},x_{2},x_{3})\in (C^{0})^3$ such that 
\begin{enumerate}
\item $c= \frac{1}{2}(\mathrm{Tr} (x_ix_j) ) $
\item $d=\mathrm{Tr}(x_{1}x_{2}x_{3})$.\label{NonEmpty2}
\end{enumerate}
\end{proposition}

\begin{proof} Let $\Xi=(a^{\prime},b^{\prime},c^{\prime},d^{\prime})\in \bold{F}^{-1}(\xi)\cap\Omega$.

Since $\Xi\in \bold{F}^{-1}(\xi)$, it follows that $\Xi=(1,b^{\prime},c^{\prime},d)$, where
\begin{align*}
b^{\prime}=&J(x_{1},x_{2},x_{3}), \text{ with } x_{j}\in C^{0}; \\
\bold{f}(c^{\prime}) =& c. 
\end{align*}

Since $\Xi\in\Omega$, by Proposition \ref{WRank1}, $d=N(b^{\prime})=\mathrm{Tr}(x_{1}x_{2}x_{3})$ and 
\begin{equation*}
c^{\prime}=(b^{\prime})^{\#}=\left(
\begin{smallmatrix}
-N(x_{1}) & \overline{x_{2}}\overline{x_{1}} & x_{3}x_{1}\\
x_{1}x_{2} & -N(x_{2}) & \overline{x_{3}}\overline{x_{2}}\\
\overline{x_{1}}\overline{x_{3}} & x_{2}x_{3} & -N(x_{3})
\end{smallmatrix}
\right)=\left(
\begin{smallmatrix}
x_{1}^{2} & x_{2}x_{1} & x_{3}x_{1}\\
x_{1}x_{2} & x_{2}^{2} & x_{3}x_{2}\\
x_{1}x_{3} & x_{2}x_{3} & x_{3}^{2}
\end{smallmatrix}
\right).
\end{equation*}
Thus
\begin{equation*}
c=\bold{f}(c^{\prime})=\frac{1}{2}\left(
\begin{smallmatrix}
\mathrm{Tr}(x_{1}^{2}) & \mathrm{Tr}(x_{1}x_{2}) & \mathrm{Tr}(x_{1}x_{3})\\
\mathrm{Tr}(x_{1}x_{2}) & \mathrm{Tr}(x_{2}^{2}) & \mathrm{Tr}(x_{2}x_{3})\\
\mathrm{Tr}(x_{1}x_{3}) & \mathrm{Tr}(x_{2}x_{3}) & \mathrm{Tr}(x_{3}^{2})
\end{smallmatrix}
\right).
\end{equation*}

\end{proof}

Now we describe $\bold{F}^{-1}((1,0,c,d))\cap \Omega$, where $c\in\mathcal{J}_{F}$ has rank 3 and $\mathrm{dim}C=4$. 

\begin{proposition}\label{Rk3Fiber} Assume $\dim C=4$. 
Let $\xi=(1,0,c,d)\in \mathbb{W}_{F}$ such that $c\in\mathcal{J}_{F}$ has rank 3. Then $\Omega_{\xi}= \bold{F}^{-1}(\xi)\cap \Omega$ is non-empty 
if and only if 
\begin{enumerate}
\item $\Aut(C)\cong \SO(3,c)$, 
\item $d^2=-4\det(c)$, so that $\xi$ has rank 3.
\end{enumerate}
If that is the case, then $\Omega_{\xi}$ is a principal homogeneous space for $\Aut(C)$. 
\end{proposition}

\begin{proof} We start with a lemma.

\begin{lemma} Let $x=(x_1,x_2,x_3)\in (C^0)^{3}$. Then we have the 
following identity of sextic polynomials 
\[ 
-4\det (\frac{1}{2} \mathrm{Tr}(x_ix_j) )= [ \mathrm{Tr}(x_1x_2x_3)]^2. 
\]
\end{lemma} 
\begin{proof} Let $g\in \GL_3(F)$. Let $y=(y_1,y_2,y_3)\in (C^{0})^{3}$ defined by $y=xg$. It is clear that 
\[ 
\det ( \mathrm{Tr}(y_iy_j) ) = \det(g)^2\cdot  \det ( \mathrm{Tr}(x_ix_j) ).
\] 
On the other hand, $\mathrm{Tr}(x_1x_2x_3)$ is a nontrivial trilinear, skew-symmetric form. Since $C^{0}$ has dimension 3, the form induces an isomorphism of $\wedge^3 C^{0}$ and $F$. 
Hence 
\[ 
\mathrm{Tr}(y_1y_2y_3)  =\det(g)\cdot  \mathrm{Tr}(x_1x_2x_3). 
\]
Since $\GL_3(F)$ acts transitively on the open set of all bases $(x_1,x_2,x_3)$ of $C^0$, it suffices now to check the identity on one basis of $C^{0}$. 
So let us take usual $i,j,k$ such that $i^2=a$, $j^2=b$, $ij=k$ and $k^2=-ab$.  Then both sides of the proposed identity are equal to $(2ab)^2$.
\end{proof} 

Now, the if and only if statement is a simple combination of the lemma and Proposition \ref{NonemptyFiber}. For the last statement, on the structure of the fiber, 
observe that the set of $x$ such that $c=\frac12(  \mathrm{Tr}(x_ix_j) )$ is a principal homogeneous space for $\mathrm{O}(C^0)$. Since 
 $\mathrm{O}(C^{0})\cong \mathrm{SO}(C^{0})\times \{\pm 1\}$ and $\mathrm{Tr}((-x_{1})(-x_{2})(-x_{3}))= -\mathrm{Tr}(x_{1}x_{2}x_{3})$, the additional 
 equation $d=\mathrm{Tr}(x_{1}x_{2}x_{3})$ assures that  the fiber $\Omega_{\xi}$ is a principal homogenous space for $\mathrm{SO}(C^{0})=\Aut(C)$.
\end{proof}

We shed some light on $\mathrm{Stab}_M(\Psi)$ for rank 3 characters. 
Recall that we have $\GL_3(F) \subset M$ such that $g\in \GL_3(F)$ acts on $\xi=\xi_{\Psi}= (1,0,c,d)\in \mathbb W_F$ by 
\[
(\det(g), 0,   \det(g)^{-1} g c g^{\top} , \det(g)^{-1} d).
\] 
Hence $g\in \mathrm{Stab}_M(\Psi)$ if and only if $\det(g)=1$ and $gcg^{\top}=c$. In other words the stabilizer of $\xi$ in $\GL_3(F)$ is the 
group $\mathrm{SO}(3,c)$.   Thus we have an action of $\Aut(C) \times \mathrm{SO}(3,c)$ on $\Omega_{\xi}$. Explicitly, 
$(g,h) \in \Aut(C) \times  \mathrm{SO}(3,c)$ acts on $x=(x_1,x_2,x_3)\in \Omega_{\xi}$ by 
\begin{equation*}
x\mapsto (gx_1, gx_2,gx_3) h^{\top}.
\end{equation*}

We have the following corollary to Proposition \ref{Rk3Fiber}. 

\begin{corollary}\label{Rk3Coin}
Assume $\mathrm{dim}C=4$. Let $\Psi$ be a rank $3$ character of $\overline{N}$ corresponding to $\xi= (1,0,c,d)\in \mathbb{W}_{F}$ such that 
$\mathrm{SO}(3,c) \cong \Aut(C)$, where $\mathrm{SO}(3,c) \subseteq \mathrm{Stab}_M(\Psi)$ described above. 
Then there are isomorphisms of $\Aut(C) \times \mathrm{SO}(3,c) $-modules
\begin{equation*}
\mathcal{V}_{(\overline{N},\Psi)}\cong C^{\infty}_{c}(\Omega)_{(\overline{N},\Psi)} \cong C^{\infty}_{c}(\Omega_{\xi})
 \cong C^{\infty}_{c}(\mathrm{Aut}(C))  \cong    C^{\infty}_{c}( \mathrm{SO}(3,c)) 
\end{equation*}
where the last two isomorphisms depend on a choice of a point in $\Omega_\xi$, giving identifications of $\Omega_{\xi}$ with $\Aut(C)$ and 
$\mathrm{SO}(3,c)$, and an isomorphism $\mathrm{SO}(3,c)\cong \Aut(C)$. 
\end{corollary}

\vskip 5pt 

\vskip 5pt 

Next we give the analog of Proposition \ref{Rk3Fiber}, where $C$ is a quadratic composition algebra. 

\begin{proposition}
Assume $\mathrm{dim}C=2$. Let $\xi=(1,0,c,0)\in \mathbb{W}_{F}$. If the set $\Omega_{\xi}= \bold{F}^{-1}(\xi)\cap \Omega$ is nonempty then $d=0$ and 
$c$ has rank at most one. If $c$ has rank one, then $\Omega_{\xi}$ is a principal homogeneous $\Aut(C)\cong \mathrm{O}(2)$-space, possibly with no rational points.  
\end{proposition} 
\begin{proof} This follows from Proposition \ref{NonemptyFiber} using that $C^0$ is one-dimensional. 
\end{proof} 

Now we discuss the rank $0$ case, i.e. $\Psi$ is trivial. We begin by computing $\bold{F}^{-1}(0)\cap \Omega$. Obseve that 
$\bold{F}^{-1}(0)=\mathcal{J}_{C^{0}}\oplus\mathcal{J}_{C^{0}}$. 
Recall that we have identified $M\cong {\mathrm{GSp}}(V_6)$ such that 
\[ 
\mathcal{J}_{C^{0}}\oplus\mathcal{J}_{C^{0}}\cong C^0\otimes V_6
\] 
via the map $(J(x),J(y))\mapsto (x_1,x_2,x_3,y_1,y_2,y_3)$ (Subsection \ref{HeisenbergSpace}).

\begin{proposition}\label{Rk0Fiber} Let $C$ be a composition algebra.
\begin{enumerate}
\item If $C^0$ is anisotropic then $\bold{F}^{-1}(0)\cap \Omega=\emptyset$.
\item\label{Split0Fiber} Suppose $C$ is a split quaternion algebra. Then  $\Omega_0=\bold{F}^{-1}(0)\cap \Omega$ consists of non-zero pure tensors 
\[ 
x\otimes v \in C^0\otimes V_6
\]
where $x^2=0$.   
\end{enumerate}
\end{proposition}

\begin{proof}
Let $\Xi\in \bold{F}^{-1}(0)=\mathcal{J}_{C^{0}}\oplus\mathcal{J}_{C^{0}}$, then $\Xi=(0,b,c,0)$, where 
\begin{align*}
b=&J(\beta_{1},\beta_{2},\beta_{3})\in \mathcal{J}_{C^{0}},\\
c=&J(\gamma_{1},\gamma_{2},\gamma_{3})\in \mathcal{J}_{C^{0}}.
\end{align*}

If $\Xi\in \Omega$, then by Lemma \ref{WRank1} we know that $b$ or $c$ is not equal to $0$ and 
\begin{align*}
b^{\#}=&0,\\
c^{\#}=&0,\\
b*c=&0.
\end{align*}

The equation $b^{\#}=0$ implies that $\beta_{i}^{2}=\beta_{j}^{2}=\beta_{i}\beta_{j}=0$. 
Similarly, the equation $c^{\#}=0$ implies that $\gamma_{i}^{2}=\gamma_{j}^{2}=\gamma_{i}\gamma_{j}=0$. If $C^0$ is anisotropic then there are no 
non-zero nilpotent elements. This proves the first claim. Now assume $C=\mathrm{M}(2,F)$, the algebra of $2\times 2$ matrices. 
The equation $b\ast c=0$ implies $\beta_i\gamma_{j}+\gamma_{j}\beta_i=0$ for all $i$ and $j$.  We need the following lemma. 

\begin{lemma} Let $\beta,\gamma \in \mathrm{M}(2,F)$. If $\beta^2=0$, $\gamma^2=0$ and $\beta\gamma+\gamma\beta=0$. Then $\beta$ and $\gamma$ are  proportional. 
\end{lemma} 
\begin{proof} This is trivial if $\beta$ or $\gamma$ is $0$, so suppose not. Then $\ker\beta=\mathrm{Im}\beta$ and $\ker\gamma=\mathrm{Im}\gamma$ are one dimensional. The equation $\beta\gamma=-\gamma\beta$ implies that $\gamma$ acts on $\ker\beta$. Thus $\ker\beta=\mathrm{Im}\beta=\ker\gamma=\mathrm{Im}\gamma$ and so $\beta$ and $\gamma$ are proportional.
\end{proof}

It follows that all $\beta_{i}$ and $\gamma_{j}$ are linearly dependent, proving the proposition.
\end{proof}
Now we can describe $\mathcal{V}_{\overline{N}}$.

\begin{proposition}\label{Rk0Coin}
Let $C$ be a composition algebra over $F$.  
\begin{enumerate}
\item If $C^0$ is anisotropic, then $\mathcal V_{\overline N}\cong \mathcal V_{\overline{\mathcal N}}$. \label{Rk0Aniso}
\item If $C=\mathrm{M}(2,F)$, write $\mathscr{G}$ for $\PGL_2(F)=\Aut(C)$, 
  then  $\mathcal V_{\overline N}$ has a composition series with a quotient $\mathcal V_{\overline{\mathcal N}} $ and a submodule 
\begin{equation*}
\mathrm{Ind}_{\mathscr{B}\times Q}^{\mathscr{G}\times \mathrm{GSp}_6} (C_c^{\infty}(F^{\times})) \otimes |\mathrm{sim}|^{3} 
\end{equation*}
where $\mathscr{B}$ is a Borel subgroup of $\mathscr{G}$, and $Q$ is a maximal parabolic in $\mathrm{GSp}_6(F)$ stabilizing a line in $V_6$, and 
 $\mathrm{sim}$ is the similitude character of $\mathrm{GSp}_6$.  The induction is not normalized. 
\end{enumerate}
\end{proposition}
\begin{proof} By Theorem \ref{HeisJacThm},  $\mathcal V_{\overline{N}}$ has a filtration with quotient $\mathcal V_{\overline{\mathcal N}}$ and submodule 
\[ 
 C_c^{\infty}(\Omega)_{\overline{N}}\cong  C_c^{\infty}(\Omega_0).
\] 
Now we apply by Proposition \ref{Rk0Fiber}. If $C^0$ is anisotropic then $\Omega_0$ is empty and we are done. So suppose $C=\mathrm{M}(2,F)$. Then 
$\Omega_0\subset C^0\otimes V_6$ consists of non-zero pure tensors $x\otimes v$ such that $x^2=0$. Fix $\omega=x\otimes v$. 
 The stabilizer in $\mathscr{G}$ of the line through $x$ is a Borel subgroup $\mathscr{B}$, and 
 the stabilizer in $\mathrm{GSp}_6$ of the line through $v$ is a maximal parabolic subgroup $Q$. 
 The stabilizer of $x\otimes v$ is a subgroup  of $\mathscr{B}\times Q$ such that the quotient is $F^{\times}$.  
 Observe that $C_c^{\infty}(\Omega_0)$, as 
a $\mathscr{G}\times   \mathrm{GSp}_6(F)$-module, is obtained by compact induction of the trivial representation of the stabilizer of $x\otimes v$. 
 Hence, using induction in stages, 
 \[ 
  C_c^{\infty}(\Omega_0)\cong \mathrm{Ind}_{\mathscr{B}\times Q}^{\mathscr{G}\times \mathrm{GSp}_6} (C_c^{\infty}(F^{\times})) 
 \]
where the induction is not normalized. This completes the proof, after taking into account additional twisting by the character of $\mathcal M$ in 
Theorem \ref{HeisJacThm}. 
\end{proof}  

\vskip 5pt 
We remark that the variant of the previous proposition, when $C$ is an octonion algebra, was obtained in \cite{SW15}.

\subsection{Fourier-Jacobi Functor}\label{FJ}

Now we recall the definition of the Fourier-Jacobi functor. (For more details, see Weissman \cite{W03}.) By the Stone-Von-Neumann theorem, the group $N$ has a unique irreducible smooth representation with central character $\psi$, denoted by $(\rho^{N}_{\psi},W_{\psi})$. By \cite[Proposition 2.5]{W03}, there is a unique extension of $(\rho^{N}_{\psi},W_{\psi})$ to a projective representation of $M_{1}N$, where $M_{1}$ is the commutator subgroup of $M$.  Furthermore, $\widetilde{\mathrm{Sp}}(d,F)$, the two-fold cover of the symplectic group $\mathrm{Sp}(d,F)$ where $d=\mathrm{dim}(N/Z)$, also acts on $W_{\psi}$ via the Weil representation. So, $\widetilde{M}_{1}\cong \widetilde{\mathrm{Sp}}(6,F)$ the metaplectic double cover of $M_{1}\cong \mathrm{Sp}(6,F)$ acts on $W_{\psi}$ through the Weil representation.

If $(\pi,V)$ is a smooth representation of $G$, then the Fourier-Jacobi functor with respect to the Heisenberg parabolic $P$ sends $\pi$ to 
\begin{equation*}
\mathrm{FJ}(\pi)=\mathrm{Hom}_{N}(W_{\psi},V_{(Z,\psi)}).
\end{equation*}
The space $\mathrm{FJ}(\pi)$ is an $\widetilde{M_{1}}$-module with the action defined by $[m\cdot f](w)=\pi(m)f(m^{-1}w)$, where the action of $\widetilde{M}$ on $V_{(Z,\psi)}$ factors through $M$. The Fourier-Jacobi functor does not depend on $\psi$ (\cite[Proposition 3.1]{W03}). Let $P_{1}=M_{1}N$. 

\vskip 5pt 
\noindent 
\textbf{Remark:} The work of Weissman \cite{W03} assumes that the groups involved are simply-laced. However, the results that we require also hold for the non-simply-laced group $F_{4}$. In particular we use \cite[Corollary 6.1.4]{W03}, which states that if the Fourier-Jacobi functor kills an irreducible representation then that representation is the trivial representation. In fact, this statement holds outside of type $C_{n}$. 

\vskip 5pt 

We use the sequence (\ref{HeisJac}) and the Fourier-Jacobi functor to investigate the constituents of $\Theta(\tau)$ via the surjection $\mathcal{V}\twoheadrightarrow \tau\otimes \Theta(\tau)$. This is done in two steps. First, we use the Fourier-Jacobi functor in conjunction with a classical theta correspondence to show that $\Theta(\tau)$ has at most two nontrivial constituents. Second, by applying twisted coinvariants to the sequence (\ref{HeisJac}) along with another application of the Fourier-Jacobi functor we show that $\Theta(\tau)$ has a single constituent, which is nontrivial, i.e. $\Theta(\tau)$ is nontrivial and irreducible.

\section{Lifting Supercuspidal representations from $\mathrm{Aut}(C)$ to $F_{4}$}\label{E7SuperCusp}

Our objective in this section is to investigate $\Theta(\tau)$, where $\tau$ is a supercuspidal representation of $\mathscr{G}$, using the tools of Section \ref{JMI}. The main result is Theorem \ref{SuperCuspLift}, where we show that $\Theta(\tau)$ is irreducible, and $\Theta(\tau_{1})\cong \Theta(\tau_{2})$ implies that $\tau_{1}\cong \tau_{2}$, where $\tau_{1}$, and $\tau_{2}$ are supercuspidal.

\subsection{At most two nontrivial constituents}\label{AtMost2NonTriv} We begin by using the Fourier-Jacobi functor and a classical theta correspondence to show that $\Theta(\tau)$ has at most two nontrivial constituents. This is accomplished in Corollary \ref{AtMost2}.

 \begin{lemma}\label{MinZTwist}
 The natural map $\mathrm{Hom}_{\mathcal{N}}(\rho^{\mathcal{N}}_{\psi},\mathcal{V}_{(Z,\psi)})\otimes \rho^{\mathcal{N}}_{\psi}\rightarrow \mathcal{V}_{(Z,\psi)}$ defines an isomorphism $\mathcal{V}_{(Z,\psi)}\cong\rho_{\psi}^{\mathcal{N}}$ of $\mathcal{P}_{1}$-modules, where $\rho_{\psi}^{\mathcal{N}}$ is the unique irreducible smooth representation of $\mathcal{N}$ with central character $\psi$.
 \end{lemma}
 
 \begin{proof} This follows from the definition of minimality in Gan-Savin \cite[Definition 3.6]{GS05} and Weissman \cite[Proposition 3.2]{W03}.\end{proof}

 \begin{lemma}\label{MinZTwist2}
 As an $\mathrm{Aut}(C)\times P_{1}$-module,
 \begin{equation*}
 \mathcal{V}_{(Z,\psi)} \cong \rho_{\psi}^{N}\otimes \omega_{\psi},
 \end{equation*}
 where $\omega_{\psi}$ is the Weil representation of $O(C^{0})\times \widetilde{\mathrm{Sp}}(V_{c})$ as a dual pair in $\widetilde{\mathrm{Sp}}(C^{0}\otimes V_{6})$. Under this isomorphism, $\mathscr{G}$ acts on the second factor, while the action of $M_{1}$ is on both factors. (We note that $M_{1}$ does not act on either factor individually. Rather $\widetilde{M}_{1}$ acts genuinely on both factors, thus the diagonal action factors through $M_{1}$.)
 \end{lemma}
 
\begin{proof}
By Lemma \ref{MinZTwist}, $\mathcal{V}_{(Z,\psi)}\cong \rho^{\mathcal{N}}_{\psi}$ as $\mathcal{P}_{1}$-modules. We must describe the restriction to $\mathscr{G}\times M_{1}$.

Let $N^{\perp}\subseteq \mathcal{N}$ be the subgroup containing $Z$ such that $N^{\perp}/Z$ is the orthogonal complement of the symplectic subspace $N/Z\subseteq \mathcal{N}/Z$. By Moeglin-Vign\'{e}ras-Waldspurger \cite[Chapitre 2, I.6 (2) and II.1 (6)]{MVW87}, it follows that $\rho_{\psi}^{\mathcal{N}}\cong \rho_{\psi}^{N}\otimes \rho_{\psi}^{N^{\perp}}$ as $\widetilde{\mathrm{Sp}}(14,F)N\times \widetilde{\mathrm{Sp}}(18,F)$-modules. Note that $\mathscr{G}$ acts trivially on $N$ and so it acts trivially on $\rho_{\psi}^{N}$.
 
Recall from Subsection \ref{HeisenbergSpace} the identification of $\mathcal{N}/Z$ with $\mathbb{W}_{C}$. This then identifies $N/Z$ with $\mathbb{W}_{F}$, and $N^{\perp}/Z$ with $\mathcal{J}_{C^{0}}\oplus \mathcal{J}_{C^{0}}$. From the isomorphism $\mathcal{J}_{C^{0}}\oplus \mathcal{J}_{C^{0}}\cong C^{0}\otimes V_{6}$ of symplectic spaces (line (\ref{SympId})), we see that the action of $\mathscr{G}\times M_{1}$ on $\rho_{\psi}^{N^{\perp}}$ is through the action of the Weil representation $\omega_{\psi}$ of $\widetilde{\mathrm{Sp}}(C^{0}\otimes V_{6})$ restricted to $O(C^{\circ})\times \widetilde{\mathrm{Sp}}(V_{6})$.
\end{proof}

Using Lemma \ref{MinZTwist2} we show in the next proposition that the Fourier-Jacobi functor with respect to $P$ applied to $\mathcal{V}$ is isomorphic to the Weil representation. This allows us to study $\Theta(\tau)$ using a classical $O(3)\times \mathrm{Sp}(6)$ theta correspondence.

\begin{proposition}\label{FJMin}
The $\mathscr{G}\times \widetilde{M}_{1}$-module $\mathrm{FJ}(\mathcal{V})$ is isomorphic to the Weil representation $\omega_{\psi}$ of $\widetilde{\mathrm{Sp}}(C^{0}\otimes V_{6})$ restricted to $\mathscr{G}\times \widetilde{M_{1}}$. (Recall that $\mathscr{G}\cong SO(C^{0})$ and $M_{1}\cong \mathrm{Sp}(6,F)$.)
\end{proposition}

\begin{proof} By definition $\mathrm{FJ}(\mathcal{V})=\mathrm{Hom}_{N}(\rho^{N}_{\psi},\mathcal{V}_{(Z,\psi)})$. By Lemma \ref{MinZTwist2}, $\mathcal{V}_{Z,\psi}\cong \omega_{\psi}\otimes \rho^{N}_{\psi}$ as $\mathscr{G}\times M_{1}N$-modules. Thus as $\mathscr{G}\times \widetilde{M}_{1}$-modules
\begin{equation*}
\mathrm{FJ}(\mathcal{V})\cong \mathrm{Hom}_{N}(\rho^{N}_{\psi},\omega_{\psi}\otimes\rho^{N}_{\psi}).
\end{equation*}
Since $\rho^{N}_{\psi}$ is a finitely generated $N$-module, $\mathrm{Hom}_{N}(\rho^{N}_{\psi},\omega_{\psi}\otimes\rho^{N}_{\psi})\cong \omega_{\psi}\otimes \mathrm{Hom}_{N}(\rho^{N}_{\psi},\rho^{N}_{\psi})$. By Schur's lemma $\mathrm{Hom}_{N}(\rho^{N}_{\psi},\rho^{N}_{\psi})\cong \mathbb{C}$. Thus $\mathrm{FJ}(\mathcal{V})\cong \omega_{\psi}$ as $\mathscr{G}\times \widetilde{M}_{1}$-modules. \end{proof}

Now we introduce some notation to discuss the $O(C^{0})\times \mathrm{Sp}(6,F)$ theta correspondence. Since $\mathscr{G}\times M_{1}\cong \mathrm{SO}(C^{0})\times \mathrm{Sp}(6,F)$, we almost have a classical dual pair. The representation $\tau$ admits two extensions to the group $O(C^{0})\cong \mathrm{SO}(C^{0})\times \{\pm id_{C^{0}}\}$ determined by whether $-id_{C^{0}}$ acts by $\pm1$. We write $\tau^{\pm}$ for the two extensions and $\Theta^{\dagger}(\tau^{\pm})$ for the big theta lift of $\tau^{\pm}$ with respect to the action of $O(C^{0})\times \widetilde{\mathrm{Sp}}(6,F)$ on the Weil representation $\omega_{\psi}$ of $\widetilde{\mathrm{Sp}}(18,F)$. 

\begin{proposition}\label{FJThetaSubQuo}
Let $\tau\in \mathrm{Irr}(\mathscr{G})$. There is a surjective $\widetilde{M}_{1}$-module homomorphism 
\begin{equation*}
\Theta^{\dagger}(\tau^{+})\oplus \Theta^{\dagger}(\tau^{+})\twoheadrightarrow\mathrm{FJ}(\Theta(\tau)).
\end{equation*}
\end{proposition}

\begin{proof} We apply the Fourier-Jacobi functor, which is exact, to the surjective map $\mathcal{V}\twoheadrightarrow\tau\otimes\Theta(\tau)$ to get a map of $\mathscr{G}\times\widetilde{M}$-modules
\begin{equation}\label{FJSurj1}
\mathrm{FJ}(\mathcal{V})\twoheadrightarrow \tau\otimes \mathrm{FJ}(\Theta(\tau)).
\end{equation} 
By Proposition \ref{FJMin}, we know that $\mathrm{FJ}(\mathcal{V})\cong \omega_{\psi}$ as $\mathscr{G}\times \widetilde{M}_{1}$-modules. 
Therefore, we have a surjective $O(C^{0})\times \widetilde{\mathrm{Sp}}(6,F)$-module map 
\begin{equation*}
\omega_{\psi}\twoheadrightarrow (\tau^{+}\otimes \Theta^{\dagger}(\tau^{+}))\oplus (\tau^{-}\otimes \Theta^{\dagger}(\tau^{-})).
\end{equation*}
Upon restricting to $SO(C^{0})\times \widetilde{\mathrm{Sp}}(6,F)\cong \mathscr{G}\times \widetilde{M}_{1}$ we get a surjective homomorphism
\begin{equation*}
\mathrm{FJ}(\mathcal{V})\cong \omega_{\psi}\twoheadrightarrow \tau\otimes (\Theta^{\dagger}(\tau^{+})\oplus \Theta^{\dagger}(\tau^{-})).
\end{equation*}
Moreover, this is the surjection onto the maximal $\tau$-isotypic quotient of $\mathrm{FJ}(\mathcal{V})$. Thus the map from line (\ref{FJSurj1}) factors through the maximal $\tau$-isotypic quotient to give a surjection
\begin{equation*}
\tau\otimes (\Theta^{\dagger}(\tau^{+})\oplus \Theta^{\dagger}(\tau^{-}))\twoheadrightarrow \tau\otimes \mathrm{FJ}(\Theta(\tau)).
\end{equation*} 
By construction, this map factors over the tensor product and the result follows.\end{proof}

\begin{corollary}\label{AtMost2}
Let $\tau\in \mathrm{Irr}(\mathscr{G})$ be a supercuspidal. The $G$-module $\Theta(\tau)$ has at most two nontrivial irreducible subquotients, each with multiplicity at most 1.
\end{corollary}

\begin{proof} This follows from Proposition \ref{FJThetaSubQuo} and the following two results. First, when $\tau^{\pm}$ is supercuspidal, the $\widetilde{M_{1}}$-modules $\Theta^{\dagger}(\tau^{+})$ and $\Theta^{\dagger}(\tau^{-})$ are irreducible and distinct (Kudla \cite{K86}; Moeglin-Vign\'{e}ras-Waldspurger \cite[Chapitre 3, IV, 4.]{MVW87}). Second, the Fourier-Jacobi functor is exact and the only irreducible representation that it kills is the trivial representation. (See \cite[Proposition 3.1; Corollary 6.1.4]{W03} and our remark in Subsection \ref{FJ}.) \end{proof}

\subsection{Unique nontrivial constituent}\label{Exactly1NonTriv}

In this subsection, we show that $\Theta(\tau)$ has exactly one nontrivial constituent. 

Using Propositions \ref{Rk3Coin} and \ref{Rk0Coin} we can compute twisted coinvariants of $\Theta(\tau)$. 

\begin{proposition}\label{ThetaPsiCo}
Let $(1,0,c, d)\in \mathbb{W}_{F}$ be an element of rank $3$ such that $\mathrm{SO}(c,3) \cong \Aut(C).$
 Let $\Psi^{\pm}$ be the character of $\overline{N}/\overline{Z}$ corresponding to the element $(1,0,c,\pm d)\in \mathbb{W}_{F}\cong N/Z$. Let $\tau\in \mathrm{Irr}(\mathscr{G})$ (not necessarily supercuspidal). 
Then as $\mathrm{SO}(c,3) \subset \mathrm{Stab}_M(\Psi^{\pm})$-modules
\begin{equation*}
\Theta(\tau)_{(\overline{N},\Psi^{\pm})}\cong\widetilde{\tau}. 
\end{equation*} 
and $\Theta(\tau)$ must have a non-trivial constituent. 

Furthermore, if $\rho_{1}$ and $\rho_{2}$ are distinct irreducible subquotients of $\Theta(\tau)$, then 
\begin{enumerate}
\item $(\rho_{j})_{(N,\Psi^{+})}\cong(\rho_{j})_{(N,\Psi^{-})}$ as $\mathrm{SO}(c,3)$-modules, $j=1,2$;
\item for $\epsilon\in \{\pm\}$, $(\rho_{1})_{(N,\Psi^{\epsilon})}$ and $(\rho_{2})_{(N,\Psi^{\epsilon})}$ cannot both be nonzero.
\end{enumerate}
\end{proposition}

\begin{proof} The first part is a simple consequence of Corollary \ref{Rk3Coin}, and Lemma \ref{ThetaDual}.

Suppose that $\rho_{1},\rho_{2}$ are two distinct irreducible subquotients of $\Theta(\tau)$. Note that the characters $\Psi^{\pm}$ are $M$-conjugate, because $s_{-1}(1,0,c,d)=(1,0,c,-d)$. Thus 
\begin{equation*}
(\rho_{j})_{(\overline{N},\Psi^{+})}\cong (\rho_{j})_{(\overline{N},\Psi^{-})}. 
\end{equation*}

Finally $(\rho_{1})_{(\overline{N},\Psi^{+})}$ and $(\rho_{2})_{(\overline{N},\Psi^{+})}$ cannot both be nonzero because this would imply that the irreducible $\mathrm{Aut}(C)$-module $\Theta(\tau)_{(\overline{N},\Psi^{+})}\cong \widetilde{\tau}$ has length greater than or equal to 2.\end{proof}

The next lemma employs two Heisenberg parabolic subgroups in $G$. Let $\overline{P}=M\overline{N}$ and $\overline{P}^{\prime}=M^{\prime}\overline{N}^{\prime}$ be two Heisenberg parabolic subgroups. Let $\overline{Z}\subset\overline{N}$ and $\overline{Z}^{\prime}\subset \overline{N}^{\prime}$ be the centers of the Heisenberg groups. Furthermore, suppose that $\overline{Z}$ ($\overline{Z}^{\prime}$) is the root subgroup associated to the $G_{2}$ relative root $2\alpha+3\beta$ ($\alpha+3\beta$).

We also use the following notation. Let $\overline{N}^{\alpha}$ be the subgroup of $\overline{N}$ generated by the root subgroups of the roots $\{2\alpha+3\beta,\alpha+3\beta,\alpha+2\beta,\alpha+\beta\}$ in the $G_{2}$ relative root system. Let $\overline{N}_{\alpha+\beta}=M^{\prime}\cap \overline{N}$, which is the root subgroup of $\alpha+\beta$. Let $L \subset M$ be the subgroup generated by elements $hs_{\mathrm{det}(h)}^{*}$, where $h\in H_{F}$. For a character $\Psi$ of $\overline{N}$, we abuse notation and continue to write $\Psi$ for its restriction to $\overline{N}^{\alpha}$ and $\overline{N}_{\alpha+\beta}$.

\begin{lemma}\label{FJFactor}
Let $\sigma$ be a smooth representation of $G$. Let $\Psi$ be the character of $\overline{N}/\overline{Z}$ corresponding to the element $(1,0,c,d)\in W_{F}$, where $\mathrm{SO}(3,c)\cong \mathscr{G}$. Then as $\mathrm{Stab}_{L}((\overline{N}^{\alpha},\Psi))=\mathrm{Stab}_{L}((\overline{N}_{\alpha+\beta},\Psi))\cong O(3,c)$-modules

\begin{equation*}
\sigma_{(\overline{N}^{\alpha},\Psi)}\cong \mathrm{FJ}^{\prime}(\sigma)_{(\overline{N}_{\alpha+\beta},\Psi)}.
\end{equation*}
\end{lemma}

\begin{proof} We begin with some preliminaries. Let $W^{+}$ be the subgroup of $\overline{N}^{\prime}$ generated by the root subgroups of the relative roots $\{\alpha+2\beta,\alpha+3\beta, 2\alpha+3\beta\}$ in the $G_{2}$ relative root system. We extend the character $\psi$ to $W^{+}$ so that it is trivial on the $\alpha+2\beta$ and $2\alpha+3\beta$ root spaces, and continue to call this extended character $\psi$. Note that this is the restriction of $\Psi$ to $W^{+}$. Now we prove the lemma.

From Weissman \cite[Proposition 3.2]{W03}, we have $\sigma_{\overline{Z}^{\prime},\psi}\cong \mathrm{Hom}_{\overline{N}^{\prime}}(\rho_{\psi}^{\overline{N}^{\prime}},\sigma_{(\overline{Z}^{\prime},\psi)})\otimes \rho_{\psi}^{\overline{N}^{\prime}}$ as $M^{\prime}_{1}\ltimes \overline{N}^{\prime}$-modules. (Remember, $\widetilde{M}_{1}^{\prime}$ acts genuinely on each factor.) It suffices for us to restrict the action of $M_{1}^{\prime}$ to the subgroup $L\cong \mathrm{GL}(3,F)$.

Since $W^{+}/Z^{\prime}$ is a maximal isotropic subspace of $N^{\prime}/Z^{\prime}$, the $(W^{+},\psi)$-coinvariants of $\rho_{\psi}^{\overline{N}^{\prime}}$ is a one dimensional space. Thus as $\mathrm{Stab}_{L}((W^{+},\psi))=L$-modules
\begin{equation*}
(\sigma_{(\overline{Z}^{\prime},\psi)})_{(W^{+},\psi)}\cong \mathrm{Hom}_{\overline{N}^{\prime}}(\rho_{\psi}^{\overline{N}^{\prime}},\sigma_{(\overline{Z}^{\prime},\psi)})=FJ^{\prime}(\sigma).
\end{equation*}

Applying $(\overline{N}_{\alpha+\beta},\Psi)$-coinvariants and using transitivity of coinvariants we get an isomorphism of $\mathrm{Stab}_{L}((\overline{N}^{\alpha},\Psi))=\mathrm{Stab}_{L}((\overline{N}_{\alpha+\beta},\Psi))\cong O(3,c)$-modules
\begin{equation*}
\sigma_{(\overline{N}^{\alpha},\Psi)}\cong [(\sigma_{(\overline{Z}^{\prime},\psi)})_{(W^{+},\psi)}]_{(\overline{N}_{\alpha+\beta},\Psi)}\cong FJ^{\prime}(\sigma)_{(\overline{N}_{\alpha+\beta},\Psi)}.
\end{equation*}
\end{proof}

\begin{proposition}\label{UniqueNonTriv}
Let $\tau\in \mathrm{Irr}(\mathscr{G})$ be supercuspidal. Then $\Theta(\tau)$ has a unique nontrivial irreducible subquotient.
\end{proposition}

\begin{proof} By Proposition \ref{ThetaPsiCo}, $\Theta(\tau)$ has at least one nontrivial irreducible subquotient.

By Proposition \ref{FJThetaSubQuo} (applied using $P^{\prime}$) we know that there is an $\widetilde{M}_{1}^{\prime}$-module surjection $\Theta^{\dagger}(\tau^{+})\oplus\Theta^{\dagger}(\tau^{-})\twoheadrightarrow FJ^{\prime}(\Theta(\tau))$. Since we are assuming that $\tau$ is supercuspidal it follows that $\Theta^{\dagger}(\tau^{\pm})$ is an irreducible $\widetilde{M}_{1}^{\prime}$-module. Thus $\mathrm{FJ}^{\prime}(\Theta(\tau))$ is completely reducible of length at most 2.

Suppose that $\Theta(\tau)$ has two distinct irreducible subquotients $\sigma^{+},\sigma^{-}$ different from the trivial representation. Since $\sigma^{\pm}$ is not trivial $\mathrm{FJ}^{\prime}(\sigma^{\pm})\neq 0$. Then without loss of generality we may assume that we have $\widetilde{M}_{1}^{\prime}$-module isomorphisms $\Theta^{\dagger}(\tau^{\pm})\cong \mathrm{FJ}^{\prime}(\sigma^{\pm})$.

We take $(\overline{N}_{\alpha+\beta},\Psi)$-coinvariants and apply Lemma \ref{FJFactor} to get $O(3,c)$-module isomorphisms $\Theta^{\dagger}(\tau^{\pm})_{(\overline{N}_{\alpha+\beta},\Psi)}\cong (\sigma^{\pm})_{(\overline{N}^{\alpha},\Psi)}$.

Now by an analog of Proposition \ref{ThetaPsiCo} in the classical case, we have $\Theta^{\dagger}(\tau^{\pm})_{(\overline{N}_{\alpha+\beta},\Psi)}\cong \widetilde{\tau}^{\pm}$ as $\mathrm{O}(3,c)$-modules.

The natural $\mathrm{SO}(3,c)$-module quotient maps $(\sigma^{\pm})_{(\overline{N}^{\alpha},\Psi)}\rightarrow (\sigma^{\pm})_{(\overline{N},\Psi^{\pm})}$ define an isomorphism $\widetilde{\tau}\cong (\sigma^{\pm})_{(\overline{N}^{\alpha},\Psi)}\rightarrow (\sigma^{\pm})_{(\overline{N},\Psi^{+})}\oplus (\sigma^{\pm})_{(\overline{N},\Psi^{-})}$ of $\mathrm{SO}(3,c)$-modules. But by Proposition \ref{ThetaPsiCo}, $(\sigma^{\epsilon})_{(\overline{N},\Psi^{+})}\oplus (\sigma^{\epsilon})_{(\overline{N},\Psi^{-})}=0$ for at least one $\epsilon\in \{\pm\}$. Thus $\widetilde{\tau}=0$, a contradiction. Therefore, $\Theta(\tau)$ must have at most one nontrivial irreducible subquotient. \end{proof}

Finally, we rule out the existence of trivial irreducible subquotients of $\Theta(\tau)$.

\begin{proposition}\label{NoTriv}
Let $\tau\in \mathrm{Irr}(\mathscr{G})$ be supercuspidal. Then $\Theta(\tau)$ does not contain an irreducible subquotient that is trivial.
\end{proposition}

\begin{proof} Suppose that the trivial representation $1$ is a subquotient of $\Theta(\tau)$, then $\tau\otimes 1_{\overline N}$ is a subquotient of $\mathcal{V}_{\overline N}$. 
Observe that  $1_{\overline N}$  is the trivial representation of $M$. 
By Proposition \ref{Rk0Coin},  $\mathcal{V}_{\overline N}$ has a $\mathscr{G}\times M$-module filtration with $\mathcal{V}_{\overline{\mathcal N}}$ as a quotient. 
From Theorem \ref{HeisJacThm},
\begin{equation*}
\mathcal{V}_{\overline{\mathcal{N}}}\cong (\mathcal{V}(\mathcal{M})\otimes |\mathrm{det}|^{3/32})\oplus |\mathrm{det}|^{6/32},
\end{equation*}
where the center of $M$ acts trivially on $\mathcal{V}(\mathcal{M})$. Thus the center of $M$ acts by two non-trivial characters on the two summands of 
$\mathcal{V}_{\overline{\mathcal{N}}}$  hence the trivial representation of $M$ cannot be a subquotient of $\mathcal{V}_{\overline{\mathcal{N}}}$. The  bottom part of the filtration, 
which appears if $C$ is split, is a principal series representation of $\mathscr{G}$, and hence $\tau$ cannot be a subquotient there.  \end{proof}

Now we prove our main theorem on the theta lift of super cuspidal representations.

\begin{theorem}\label{SuperCuspLift}
Let $\tau_1, \tau_{2}\in \mathrm{Irr}(\mathscr{G})$ where $\tau_1$ is supercuspidal. 

\begin{enumerate}
\item The theta lift $\Theta(\tau_1)$  is an irreducible representation of $G$. \label{SCIrr}
\item If $\theta(\tau_{1})\cong \theta(\tau_{2})$, then $\tau_{1}\cong \tau_{2}$.  \label{PartialSmallTheta1-1}
\end{enumerate}
\end{theorem}

\begin{proof} (\ref{SCIrr}) By Propositions \ref{UniqueNonTriv} and \ref{NoTriv} the representation $\Theta(\tau)$ is irreducible.
(\ref{PartialSmallTheta1-1}) By the assumption, and using  (\ref{SCIrr}), we have a surjection $\Theta(\tau_{2})\twoheadrightarrow \Theta(\tau_{1})$. Then, by Proposition \ref{ThetaPsiCo}, 
\begin{equation*}
\widetilde{\tau}_{2}\cong\Theta(\tau_{2})_{(N,\Psi^{\pm})}\twoheadrightarrow\Theta(\tau_{1})_{(N,\Psi^{\pm})}\cong\widetilde{\tau}_{1}.
\end{equation*}
\end{proof}

\section{Jacquet Modules II}\label{JMII}

In this section, we take $C$ to be the split quaternion algebra $M(2,F)$ of $2\times 2$ matrices with entries in $F$. In particular, $\mathscr{G}=\mathrm{Aut}(C)=\mathrm{PGL}_{2}(F)$. The main objective of this section is to compute the Jacquet module of the minimal representation of $E_{7}$ with respect to a Borel subgroup of $\mathscr{G}$. These calculations are applied in Section \ref{LiftPSeries} to compute the big theta lift of constituents of principal series of $\mathscr{G}$ to $G$.

Our approach follows the argument of Savin \cite{S94} and Magaard-Savin \cite{MS97} utilizing the exact sequence from \cite[Theorem 6.5]{S94}, which we recall after introducing some notation. 

Fix a Borel subgroup $\mathcal{B}\subset\mathcal{G}$, let $\mathcal{P}\supset \mathcal{B}$ be the unique maximal parabolic subgroup  corresponding to the $E_{6}$ sub-diagram inside the $E_{7}$ diagram. Fix a Levi decomposition $\mathcal{P}=\mathcal{M}\mathcal{N}$ and note that $\mathcal{N}$ can be given the structure of the exceptional cubic Jordan algebra $\mathcal{J}=\mathcal{J}_{\mathbb{O}}$ \cite{KS15}. We identify $\mathcal{N}$ with $\mathcal{J}$ as $F$-vector spaces. Under this identification $\mathcal{M}$ is the group of linear transformations of $\mathcal{J}$ that preserve the cubic norm form of $\mathcal{J}$ up to scaling. The semisimple part of $\mathcal{M}$ is a group of type $E_{6}$. (For details see \cite{K67}.)

 Let $\omega$ be the set of singular points in $\mathcal J\cong \mathcal{N}$, i.e. the highest weight vectors for a Borel subgroup in $\mathcal{M}$. Equivalently, $\omega$ is the set of rank $1$ elements in $\mathcal{J}$.

\begin{theorem}[Magaard-Savin \cite{MS97}, Theorem 1.1; Savin \cite{S94}, Theorem 6.5]\label{PFiltration} Let $\mathcal{P}=\mathcal{M}\mathcal{N}$ be the maximal parabolic subgroup defined above. Let $\overline{\mathcal{P}}=\mathcal{M}\overline{\mathcal{N}}$ be its opposite. The minimal representation $(\Pi,\mathcal{V})$ of $\mathcal{G}$ has a $\overline{\mathcal{P}}$-invariant filtration 

\begin{equation}\label{MinSeq}
0\rightarrow C^{\infty}_{c}(\omega)\rightarrow \mathcal{V}\rightarrow \mathcal{V}_{\overline{\mathcal{N}}}\rightarrow 0.
\end{equation}
Here $C^{\infty}_{c}(\omega)$ denotes the space of locally constant, compactly supported functions on $\omega$, and $\mathcal{V}_{\overline{\mathcal{N}}}$ is the space of $\overline{\mathcal{N}}$-coinvariants of $\mathcal{V}$. Furthermore, the $\overline{\mathcal{P}}$-module structure is given by:
\begin{enumerate}
\item Let $f\in C^{\infty}_{c}(\omega)$ and let $m\overline{n}\in \overline{\mathcal{P}}=\mathcal{M}\overline{\mathcal{N}}$. Then
\begin{equation}\label{PhaseTrans}
[\Pi(\overline{n})f](x)=\psi(\langle x,\overline{n}\rangle)f(x)
\end{equation}
and
\begin{equation}\label{ScaleGeo}
[\Pi(m)f](x)=|\mathrm{det}(m)|^{s/d}f(m^{-1}x).
\end{equation}
\item 
\begin{equation}\label{MinCoIn}
\mathcal{V}_{\overline{\mathcal{N}}}\cong \mathcal{V}(\mathcal{M})\otimes |\mathrm{det}|^{t/d}+|\mathrm{det}|^{s/d},
\end{equation}
where $\mathcal{V}(\mathcal{M})$ is the minimal representation of $\mathcal{M}$ (center acting trivially).
\end{enumerate}

Above $\langle-,-\rangle:\mathcal{N}\times \overline{\mathcal{N}}\rightarrow F$ is the $F$-valued pairing induced by the Killing form on $\mathrm{Lie}(\mathcal{G})$, and $\mathrm{det}$ is the determinant of the representation of $\mathcal{M}$ on $\overline{\mathcal{N}}$, and $d$ is the dimension of $\mathcal{N}$. The values of $s$ and $t$ are given in the following table.

\begin{center}
\begin{tabular}{|c||c|c|}\hline
$\mathcal{G}$&$s$&$t$\\ \hline\hline
$E_{6}$&$4$&$2$\\ \hline
$E_{7}$&$6$&$3$\\ \hline
\end{tabular} 
\end{center}

\end{theorem}

It will be convenient to describe the dual pair $\mathscr{G}\times G$ in terms of the parabolic subgroup $\mathcal{P}$. If we identify $\mathcal N \cong \mathcal{J}$ 
and $\mathcal M$ with the group of similitudes of the norm form, then $G\cong \Aut( \mathcal{J})$ sits in $\mathcal M$. Let $\mathscr G$ be the centralizer 
of $G$ in $\mathcal G$. Let $\mathscr{B}=\mathscr{T}\mathscr{U}\subset \mathscr{G}$ be  Borel subgroup defined by 
\[ 
\mathscr{T}= \mathscr G \cap \mathcal M \text{ and } \mathscr{U}= \mathscr G \cap \mathcal N. 
\] 
Then $\mathscr{T}$ is the center of $\mathcal{M}$  and $\mathscr{U}$ the set of scalar matrices in $\mathcal{J}$ under the identification 
$\mathcal N \cong \mathcal{J}$. Recall that, for the purpose of describing representations of $\mathscr G$, we identified $\mathscr T$ with $F^{\times}$ such that 
$\mathscr T$ acts on $\overline{\mathscr U}$ by multiplication, it follows that $\mathscr T \cong F^{\times}$ acts on $\mathcal N \cong \mathcal{J}$ 
by inverse scalar multiplication. Thus the character $|\mathrm{det}(m)|^{s/d}$ restricted to $\mathscr T$ is $|t|^6$.

\subsection{Untwisted Jacquet modules}

Our objective in this section is to describe the $\mathscr{T}\times G$-module $r_{\overline{\mathscr{B}}}(C^{\infty}_{c}(\omega))$. This is accomplished in Proposition \ref{JMIofMin}. 

For $S\subset \overline{\mathcal{N}}$ we write
\begin{equation*}
S^{\perp}=\{x\in \mathcal{N}|\psi(\langle x,s\rangle)=1, \text{ for all }s\in S\}.
\end{equation*}
Let $\mathcal{J}^0$ be the set of trace 0 elements in $\mathcal{J}$. 
Under the identification $\mathcal N \cong \mathcal{J}$,  we have $\overline{\mathscr{U}}^{\perp}\cong \mathcal{J}^0$.
Let $\omega_{0}=\omega\cap \mathcal{J}^0$.

\begin{lemma}\label{Rk1Tr0}
The restriction map $C^{\infty}_{c}(\omega)\rightarrow C^{\infty}_{c}(\omega_{0})$ induces an isomorphism $r_{\overline{\mathscr{B}}}(C^{\infty}_{c}(\omega))\cong C^{\infty}_{c}(\omega_{0})$ of $(\mathscr{T}\times G)$-modules, where the action on $C^{\infty}_{c}(\omega_{0})$ is given by 
\begin{equation*}
((t,g)\cdot f)(x)=|t|^{6-\frac{1}{2}}f(g^{-1}\cdot x t), \hspace{1cm} (t,g)\in \mathscr{T}\times G.
\end{equation*}
\end{lemma}

\begin{proof} 
The proof is the same as Magaard-Savin \cite{MS97}, Lemma 2.2.
\end{proof}

\begin{lemma}[Aschbacher \cite{A87}, section (8.6)]
The action of $\mathscr{T}\times G$ on $\omega_{0}$ is transitive.
\end{lemma}

\begin{proof}
This follows from translating Aschbacher's terminology into ours.
\end{proof}

If $x_0\in \omega_0$, then, as $\mathscr{T}\times G$-modules
\[ 
C^{\infty}_{c}(\omega_{0})\cong |-|^{\frac{11}{2}} \cdot 
\mathrm{ind}_{\mathrm{Stab}_{\mathscr{T}\times G}(x_{0})}^{\mathscr{T}\times G} (1) 
\]

Next, we want to describe the stabilizer of a point in $\omega_0$. 
The highest weight of the action of $G$ on $\mathcal J^0\subset \mathcal{J}$ is the fundamental weight $\omega_{4}$ taking value $1$ on the simple coroot $\alpha_{4}^{\vee}$ and $0$ on the other simple coroots. (We are using the Bourbaki labeling for simple roots \cite[Plate VIII]{B02}. In particular, $\alpha_{4}$ is the short simple root of degree 1 in the Dynkin diagram.) The next lemma follows directly from definitions.

\begin{lemma}
Let $x_{0}\in \omega_{0}$ be a highest weight vector with respect to the Borel subgroup $B$. Let 
 $Q\supseteq B$ be the maximal parabolic subgroup of $G$ obtained by removing $\alpha_{4}$ from the $F_{4}$ Dynkin diagram. This yields a parabolic subgroup of type $B_{3}$. 
Then 
\[ 
\mathrm{Stab}_{\mathscr{T}\times G}(x_{0})= \{ (t,q)\in 
\mathscr{T}\times Q ~|~t=\omega_4(q) \}.
\] 
\end{lemma}

By transitivity of induction, we have the $\mathscr{T}\times G$-module isomorphism 
\begin{equation*}
\mathrm{ind}_{\mathrm{Stab}_{\mathscr{T}\times G}(x_{0})}^{\mathscr{T}\times G}(1) \cong 
\mathrm{Ind}_{\mathscr{T}\times Q}^{\mathscr{T}\times G}( \mathrm{ind}_{\mathrm{Stab}_{\mathscr{T}\times G}(x_{0})}^{\mathscr{T}\times Q}(1))
 \cong  \mathrm{Ind}_{\mathscr{T}\times Q}^{\mathscr{T}\times G} (C^{\infty}_{c}(F^{\times})).
\end{equation*}
In order to write the last module in terms of normalized parabolic induction, we need to replace $C^{\infty}_{c}(F^{\times})$ by 
$\delta_{Q}^{-1/2}\cdot C^{\infty}_{c}(F^{\times})$. Recall that  $\delta_{Q}^{1/2}(q)=|\omega_4(q)|^{-11/2}$. We also need to bring back the twist by $|t|^{11/2}$. 
Observe that the two exponents are inverses of each other. Hence 

\begin{equation}\label{MinJacIso}
r_{\overline{\mathscr{B}}}(C^{\infty}_{c}(\omega))\cong i_{\mathscr{T}\times Q}^{\mathscr{T}\times G}(C^{\infty}_{c}(F^{\times})),
\end{equation}
where the action of $\mathscr{T}\times Q$ on $C^{\infty}_{c}(F^{\times})$ is given by
\begin{equation}\label{TQAction}
((t,q)\cdot f)(x)=f(\omega_{4}(q^{-1})xt).
\end{equation}

Putting everything together we have the following description of $r_{\overline{\mathscr{B}}}(\mathcal{V})$.

\begin{proposition}\label{JMIofMin}
As a representation of $\mathscr{T}\times G$, the module $r_{\overline{\mathscr{B}}}(\mathcal{V})$ has a filtration with successive quotients 
\[
 |-|^{5/2}\cdot  \mathcal V(\mathcal M) \oplus |-|^{11/2},
 \]
 \[
i_{\mathscr{T}\times Q}^{\mathscr{T}\times G}(C^{\infty}_{c} (F^{\times})),
\]
where the action of $\mathscr{T}\times Q$ on $C^{\infty}_{c}(F^{\times})$ is given by equation (\ref{TQAction}).
\end{proposition}

\subsection{Twisted Jacquet modules} In this subsection we compute the twisted Jacquet modules of $\mathcal{V}$ with respect to $(\overline{\mathscr{U}},\psi)$. (Recall that $\overline{\mathscr{U}}\cong F$, so $\psi$ defines a character of $\overline{\mathscr{U}}$.) Note that $G=\mathrm{Stab}_{\mathcal{M}}((\overline{\mathscr{U}},\psi))$.

\begin{lemma}\label{PGL2TwistJac1}
The inclusion $C^{\infty}_{c}(\omega)\hookrightarrow \mathcal{V}$ (from line (\ref{MinSeq})) induces an isomorphism of $G$-modules.
\begin{equation}
C^{\infty}_{c}(\omega)_{(\overline{\mathscr{U}},\psi)}\cong \mathcal{V}_{(\overline{\mathscr{U}},\psi)}.
\end{equation}
\end{lemma}

\begin{proof} Apply $(\overline{\mathscr{U}},\psi)$-coinvariants, which is exact, to line (\ref{MinSeq}). Note that $(\mathcal{V}_{\overline{\mathcal{N}}})_{(\overline{\mathscr{U}},\psi)}=0$. \end{proof}

To study $C^{\infty}_{c}(\omega)_{(\overline{\mathscr{U}},\psi)}$ we need to consider the set of rank 1 trace 1 elements in $\mathcal{J}$. Viewing $\omega$ as the set of rank $1$ elements in $\mathcal{J}$, we define $\omega_{1}=\{x\in\omega|\mathrm{Tr}(x)=1\}$.

\begin{lemma}\label{PGL2TwistJac2}
The restriction map $C^{\infty}_{c}(\omega)\rightarrow C^{\infty}_{c}(\omega_{1})$ induces a $G$-module isomorphism $C^{\infty}_{c}(\omega)_{(\overline{\mathscr{U}},\psi)}\cong C^{\infty}_{c}(\omega_{1})$. (The action of $G$ on $C^{\infty}_{c}(\omega_{1})$ is the same as on line (\ref{ScaleGeo}).)
\end{lemma}

\begin{proof} This is proved as in Magaard-Savin \cite{MS97}, Lemma 2.2.\end{proof}

\begin{lemma}\label{rk1tr1Trans}
The action of $G$ on $\omega_{1}$ is transitive. Let $v_{0}=\mathrm{diag}(1,0,0)\in\mathcal{J}$, then $\mathrm{Stab}_{G}(v_{0})\cong \mathrm{Spin}(9,F)$ (split spin group). Thus the map $G/\mathrm{Stab}_{G}(v_{0})\rightarrow \omega_{1}$ defined by $ g\mathrm{Stab}_{G}(v_{0})\mapsto g\cdot v_{0}$ is a bijection.
\end{lemma}

\begin{proof}  This is Corollary 5.8.2 and Theorem 7.1.3 in Spinger-Veldkamp \cite{SV00}. We make a few remarks and match our notation with Springer-Veldkamp. 

Our $v_{0}$ is the $u$ in Springer-Veldkamp. The space $E_{0}$ in loc. cit. is then the elements of $\mathcal{J}$ of the form $\mathrm{diag}(0,b,-b)+J(x,0,0)$, where $b\in F$ and $x\in\mathbb{O}$. This is a $9$ dimensional orthogonal space where the quadratic form is the restriction of the trace form of $\mathcal{J}$ to $E_{0}$. This form on $E_{0}$ is nondegenerate. So by Springer-Veldkamp Theorem 7.1.3, we see that $\mathrm{Stab}_{G}(v_{0})\cong \mathrm{Spin}(Q,E_{0})$.\end{proof}

\textbf{Remark:} We note that the quadratic space in the previous lemma $(Q,E_{0})$ decomposes as an orthogonal sum of a one dimensional quadratic space $(Q_{0}^{\prime},F\cdot v)$, where $Q_{0}^{\prime}(v)=2$, and the $8$ dimensional quadratic space associated to the split octonian algebra $\mathbb{O}$.

\begin{lemma}\label{PGL2TwistJac3}
Let $v_{0}=\mathrm{diag}(1,0,0)\in\mathcal{J}$. The map $ind_{\mathrm{Stab}_{G}(v_{0})}^{G}(1) \rightarrow C^{\infty}_{c}(\omega_{1})$ defined by $f \mapsto (g^{-1}\cdot v_{0}\mapsto f(g))$ is an isomorphism of $G$-modules.
\end{lemma}

\begin{proof} A direct calculation shows this map is a $G$-module homomorphism. Here we are using the fact that any character of $G$ is trivial. 

One can directly check that the inverse map is given by $F\mapsto (g\mapsto F(g^{-1}\cdot v_{0}))$.\end{proof}

\begin{proposition}\label{MinRepPGL2TwistedJac}
By combining the isomorphisms of Lemmas \ref{PGL2TwistJac1}, \ref{PGL2TwistJac2}, and \ref{PGL2TwistJac3} we see that as $G$-modules
\begin{equation*}
\mathcal{V}_{(\overline{\mathscr{U}},\psi)}\cong ind_{\mathrm{Stab}(v_{0})}^{G}(1).
\end{equation*}
\end{proposition}

\section{Lifting from $\mathrm{PGL}(2)$ to $F_{4}$}\label{LiftPSeries}

We continue to use the notation from Section \ref{JMII}. In particular, $C=M(2,F)$. In this section, we explicitly describe the theta lift from $\mathscr{G}$ to $G$. 

In Subsections \ref{PSTheta} and \ref{TSTheta} we compute the big theta lift of constituents of principal series; the small theta lift is computed in Subsection \ref{PStheta}. In subsection \ref{PGL2SC}, we revisit the theta lift of supercuspidal representations.

\subsection{Principal series}\label{PSTheta}

Now we begin the calculation of $\mathrm{Hom}_{\mathscr{G}}(\mathcal{V},\tau)$, where $\tau=i_{\overline{\mathscr{B}}}^{\mathscr{G}}(\chi)$ is a principal series (not necessarily irreducible) of $\mathscr{G}$ induced from a character $\chi:\mathscr{T}\rightarrow \mathbb{C}^{\times}$. 

 By Frobenius reciprocity,

\begin{equation}\label{ThetaPSFR}
\mathrm{Hom}_{\mathscr{G}}(\mathcal{V},\tau)\cong \mathrm{Hom}_{\mathscr{T}}(r_{\overline{\mathscr{B}}}(\mathcal{V}),\chi).
\end{equation}

Note that $\mathrm{Stab}_{\overline{\mathcal{P}}}(\overline{\mathscr{U}})=(\mathscr{T}\times G)\overline{\mathcal{N}}$. Thus we apply the exact functor $r_{\overline{\mathscr{B}}}$ to sequence (\ref{MinSeq}) to get a sequence of $\mathscr{T}\times G$-modules

\begin{equation}\label{MinSeqU}
0\rightarrow r_{\overline{\mathscr{B}}}(C^{\infty}_{c}(\omega))\rightarrow r_{\overline{\mathscr{B}}}(\mathcal{V})\rightarrow \delta_{\overline{\mathscr{B}}}^{-1/2}\otimes\mathcal{V}_{\overline{\mathcal{N}}}\rightarrow 0.
\end{equation}
We apply the functor $\mathrm{Hom}_{\mathscr{T}}(-,\chi)$ to (\ref{MinSeqU}) to get a long exact sequence. Let $X,Y,Z$ be the nonzero $\mathscr{T}\times G$-modules in sequence (\ref{MinSeqU}) from left to right, respectively. Then the long exact sequence is:

\begin{equation}\label{LES}
\begin{tikzcd}
0\arrow[r]
& \mathrm{Hom}_{\mathscr{T}}(Z,\chi) \arrow[r]
& \mathrm{Hom}_{\mathscr{T}}(Y,\chi) \arrow[r,"\iota"]
& \mathrm{Hom}_{\mathscr{T}}(X,\chi) & \\
\phantom{A}\arrow[r] 
& \mathrm{Ext}_{\mathscr{T}}^{1}(Z,\chi) \arrow[r]
& \mathrm{Ext}_{\mathscr{T}}^{1}(Y,\chi) \arrow[r]
& \mathrm{Ext}_{\mathscr{T}}^{1}(X,\chi) \arrow[r] &...\\
\end{tikzcd}
\end{equation}

The following lemma shows that $\iota$ is an isomorphism if $\chi$ avoids a finite set of characters. It is a simple consequence of Theorem \ref{PFiltration}.

\begin{lemma}\label{0Hom}
Assume that $\chi\neq |\cdot|^{5/2}$ and $|\cdot|^{11/2}$. Then

\begin{align*}
\mathrm{Hom}_{\mathscr{T}}(\delta_{\overline{\mathscr{B}}}^{-1/2}\otimes \mathcal{V}_{\overline{\mathcal{N}}},\chi)=&0\text{ and }\\
\mathrm{Ext}^{1}_{\mathscr{T}}(\delta_{\overline{\mathscr{B}}}^{-1/2}\otimes \mathcal{V}_{\overline{\mathcal{N}}},\chi)=&0,
\end{align*}
and the map $\iota$ in the long exact sequence (\ref{LES}) induces an isomorphism
\begin{equation*}
\mathrm{Hom}_{\mathscr{T}}(r_{\overline{\mathscr{B}}}(\mathcal{V}),\chi)\cong \mathrm{Hom}_{\mathscr{T}}(r_{\overline{\mathscr{B}}}(C^{\infty}_{c}(\omega)),\chi).
\end{equation*}
\end{lemma}

Now we can prove the main result of this subsection.

\begin{theorem}\label{ThetaPS} 
Let $\chi$ be a character of $\mathscr{T}$ such that $\chi\neq |\cdot|^{-5/2}$ and $|\cdot|^{-11/2}$.
Let $\pi$ be an irreducible quotient of $i_{\overline{\mathscr{B}}}^{\mathscr{G}}(\chi)$. Then $\Theta(\pi)$ is a quotient of 
$i_{Q}^{G}(\chi\circ\omega_{4})$. Moreover, if $i_{\overline{\mathscr{B}}}^{\mathscr{G}}(\chi)$ is irreducible, then 
$\Theta(i_{\overline{\mathscr{B}}}^{\mathscr{G}}(\chi))= i_{Q}^{G}(\chi\circ\omega_{4})$.   

\end{theorem}

\begin{proof} Observe that $\pi$ is a submodule of $i_{\overline{\mathscr{B}}}^{\mathscr{G}}(\chi^{-1})$, since $\pi$ is self-dual. 
By Lemma \ref{ThetaDual},
\[ 
\Theta(\pi)^{*} \cong \mathrm{Hom}_{\mathscr{G}}(\mathcal{V},\pi)\subseteq \mathrm{Hom}_{\mathscr{G}}(\mathcal{V},i_{\overline{\mathscr{B}}}^{\mathscr{G}}(\chi^{-1}))
\] 
and we are going to compute the latter space. 
By Frobenius reciprocity, 
\begin{equation*}
\mathrm{Hom}_{\mathscr{G}}(\mathcal{V},i_{\overline{\mathscr{B}}}^{\mathscr{G}}(\chi^{-1}))\cong \mathrm{Hom}_{\mathscr{T}}(r_{\overline{\mathscr{B}}}(\mathcal{V}),\chi^{-1}).
\end{equation*}
Since $\chi^{-1} \neq |\cdot|^{5/2}$ and $|\cdot|^{11/2}$ we can apply Lemma \ref{0Hom} to get 

\begin{equation*}
\mathrm{Hom}_{\mathscr{T}}(r_{\overline{\mathscr{B}}}(\mathcal{V}),\chi^{-1})\cong \mathrm{Hom}_{\mathscr{T}}(r_{\overline{\mathscr{B}}}(C^{\infty}_{c}(\omega)), \chi^{-1}).
\end{equation*}
By equation (\ref{MinJacIso}) we have

\begin{equation*}
\mathrm{Hom}_{\mathscr{T}}(r_{\overline{\mathscr{B}}}(C^{\infty}_{c}(\omega)), \chi^{-1})\cong \mathrm{Hom}_{\mathscr{T}}(i_{\mathscr{T}\times Q}^{\mathscr{T}\times G}(C^{\infty}_{c}(F^{\times})),\chi^{-1}).
\end{equation*}

The maximal $\chi^{-1}$-isotypic quotient of the $\mathscr{T}\times Q$-module $C^{\infty}_{c}(F^{\times})$ is 
$\chi^{-1}\otimes\chi\circ\omega_{4}$. Thus by \cite[Lemma 9.4]{GG06},

\begin{equation*}
\mathrm{Hom}_{\mathscr{T}}(i_{\mathscr{T}\times Q}^{\mathscr{T}\times G}(C^{\infty}_{c}(F^{\times})),\chi^{-1})\cong \mathrm{Hom}_{\mathbb{C}}(i_{Q}^{G}(\chi\circ\omega_{4}),\mathbb{C})=i_{Q}^{G}(\chi\circ\omega_{4})^{*}.
\end{equation*}

\end{proof}

Any irreducible non-supercuspidal representation of $\mathscr G$ is either an irreducible quotient of $i_{\overline{\mathscr{B}}}^{\mathscr{G}}(\chi)$, where $|\chi|=|-|^s$ with $s\geq 0$, 
or it is a quadratic twist of Steinberg. But these representations are quotients of $i_{\overline{\mathscr{B}}}^{\mathscr{G}}(\chi)$ such that $|\chi|=|\cdot |^{-1/2}$, so Theorem \ref{ThetaPS} applies to all irreducible non-supercuspidal representations. However, it does not provide a full understanding of the big theta lift of constituents of reducible principal series. We resolve this point in the next subsection.

\subsection{Trivial and Steinberg}\label{TSTheta}

In this subsection we study the theta lifts of the trivial and Steinberg representations of $\mathscr{G}$, along with their twists.

\begin{theorem} \label{ThetaPSsubQ}
Let $\chi$ be a character of $\mathscr{T}$ such that $\chi=\chi_0 |-|^{1/2}$, where $\chi_0$ is a quadratic character.  Then 
\begin{enumerate}
\item\label{ThetaChar}
$\Theta(\chi_0)$ is the unique irreducible quotient of $i_{Q}^{G}(\chi\circ\omega_{4})$.
\item\label{ThetaSt} 
$\Theta(\mathrm{St}\otimes \chi_0)$ is the unique irreducible submodule of $i_{Q}^{G}(\chi\circ\omega_{4})$.
\end{enumerate}
\end{theorem}

\begin{proof} We already know that $\Theta(\chi_0)$ is a quotient of $i_{Q}^{G}(\chi\circ\omega_{4})$ and 
$\Theta(\mathrm{St}\otimes \chi_0)$ is a quotient of $i_{Q}^{G}(\chi^{-1}\circ\omega_{4})$, which is the same as a submodule of  $i_{Q}^{G}(\chi\circ\omega_{4})$. 
(The representations $i_{Q}^{G}(\chi\circ\omega_{4})$ and $i_{Q}^{G}(\chi^{-1}\circ\omega_{4})$ each have length $2$. Moreover, the irreducible sub of one is the quotient of the other \cite[Theorem 6.1]{CJ10}.) 

 We work with both cases simultaneously. 
By Proposition \ref{ThetaPsiCo}, 
\[ 
\Theta(\chi_0)_{(\overline{N},\Psi)}\cong \chi_0^{-1} \text{ and } 
\Theta(\mathrm{St}\otimes \chi_0)_{(\overline{N},\Psi)}\cong \mathrm{St}\otimes \chi_0^{-1}.
\] 
 Thus $\chi_0^{-1}$ is a quotient of $i_{Q}^{G}(\chi\circ\omega_{4})_{(\overline{N},\Psi)}$ while $\mathrm{St}\otimes \chi_0^{-1}$ is a submodule. This implies that 
 neither $\Theta(\chi_0)$ nor $\Theta(\mathrm{St}\otimes \chi_0)$ could be isomorphic to $i_{Q}^{G}(\chi\circ\omega_{4})$. \end{proof}

\subsection{Small Theta}\label{PStheta}

In this section we describe $\theta(\pi)$, where $\pi$ is a constituent of a principal series of $\mathscr{G}$. The next proposition follows from Propositions \ref{CJDPS}, \ref{ThetaPSsubQ}, and Theorem \ref{ThetaPS}.

\begin{proposition}\label{SmallThetaPS}
Let $\chi$ be a character of $\mathscr{T}$ so that $\chi=|-|^{s}\cdot \chi_{0}$, where $s\geq 0$ and $\chi_{0}$ is a unitary character of $\mathscr{T}$.
\begin{enumerate}
\item If $s\neq\frac{1}{2}$ or $\chi_{0}$ is not of order dividing $2$, then $i_{\overline{\mathscr{B}}}^{\mathscr{G}}(\chi)$ is irreducible, and
\begin{enumerate}
\item if $s\neq\frac{5}{2},\pm\frac{11}{2}$ or $\chi_{0}$ is not trivial, then $i_{Q}^{G}(\chi\circ\omega_{4})$ is irreducible, so $\theta(i_{\overline{\mathscr{B}}}^{\mathscr{G}}(\chi))=\Theta(i_{\overline{\mathscr{B}}}^{\mathscr{G}}(\chi))\cong i_{Q}^{G}(\chi\circ\omega_{4})$;
\item if $s=\frac{11}{2}$ and $\chi_{0}$ is trivial, then $\theta(i_{\overline{\mathscr{B}}}^{\mathscr{G}}(\chi))$ is the unique irreducible quotient of $i_{Q}^{G}(|-|^{\frac{11}{2}}\circ\omega_{4})$, 
which is the trivial representation of $G$.  
\item if $s=\frac{5}{2}$ and $\chi_{0}$ is trivial, then $\theta(i_{\overline{\mathscr{B}}}^{\mathscr{G}}(\chi))$ is the unique semisimple quotient of $i_{Q}^{G}(|-|^{\frac{5}{2}}\circ\omega_{4})$, which has the form $\sigma^+\oplus \sigma^-$ where $\sigma^+$ and $\sigma^-$ are distinct irreducible representations of $G$.

\end{enumerate}
\item If $s=\frac{1}{2}$ and $\chi_{0}$ has order dividing $2$, then: 
\begin{enumerate}
\item $\theta(\chi_0)=\Theta(\chi_0)$ is the unique irreducible quotient of $i_{Q}^{G}(\chi\circ\omega_{4})$;
\item $\theta(\mathrm{St}\otimes \chi_0)=\Theta(\mathrm{St}\otimes \chi_0)$ is the unique irreducible submodule of $i_{Q}^{G}(\chi\circ\omega_{4})$.
\end{enumerate}
\end{enumerate}
\end{proposition}

\subsection{Supercuspidal representations}\label{PGL2SC}

In this subsection we revisit the theta lift of supercuspidal representations of $\mathrm{PGL}_{2}(F)$. This calculation involves the $F_{4}\times G_{2}$ dual pair inside of $E_{8}$. 

For this subsection we maintain our previous notation with the following exceptions. We redefine $P$, $M$, and $N$ below. We write $(\Pi_{n},\mathcal{V}_{n})$ for the the minimal representation of $E_n$.

Let $\tau$ be a supercuspidal representation of $\mathrm{PGL}_{2}(F)$. Then $\sigma:=\Theta(\tau)$ is irreducible by Theorem \ref{SCIrr}. Let $Q_2\subset G$ be the 
maximal parabolic that stabilizes a 2-dimensional singular (also called amber) subspace in the 26-dimensional representation.  The standard $Q_2$ (corresponding to a fixed choice of positive roots) is the stabilizer of the amber space spanned by the weights $\omega_4$ and $\omega_4-\alpha_4$.  
We note that the Levi of $Q_2$ has type $A_{2,\mathrm{long}} \times A_{1,\mathrm{short}}$. Observe that 
$Q_2$ has a quotient isomorphic to $\GL_2$ given by the action of $Q_2$ on the stabilized amber space. 
With this identification, $\det$ can be naturally viewed as a character of $Q_2$, and $\tau$ can be inflated to $Q_2$.  
The modular character is $\rho_{Q_2}(g)= |\det(g)|^{7/2}$.  We have the following: 

\begin{proposition}\label{ThetaSCasQuotient} $\sigma$ is the unique irreducible quotient of $\Ind_{Q_2}^G( \tau \otimes |\det|^{3/2})$. 
\end{proposition}
\begin{proof}  Let $P=MN\subset G_2$ 
be the Heisenberg parabolic. Then $M\cong \GL_2$ and $G\times \GL_2$ is a subgroup of the Levi factor $E_7$ in $E_8$ such that the 
quotient by the center of the Levi gives the dual pair $G\times \PGL_2$ in the adjoint $E_7$.  
In Magaard-Savin \cite[Theorem 7.6]{MS97},  $r_P(\mathcal{V}_{8})$ was shown to have a $G\times \GL_2$-module filtration with three pieces. 
The top (quotient) is
\[ 
\mathcal{V}_{7} \otimes |\det|^{3/2} \oplus 1\otimes |\det|^{7/2}. 
\] 
Since $\sigma\otimes \tau$ is a quotient of $\mathcal{V}_{7}$, by Frobenius reciprocity, $\sigma \otimes \Ind_P^{G_2} (\tau \otimes |\det|^{3/2})$ is a quotient of 
$\mathcal{V}_{8}$. Here we are using that $\Ind_P^{G_2} (\tau \otimes |\det|^{s})$ reduces only for $s=\pm 1/2$, in particular, $\Ind_P^{G_2} (\tau \otimes |\det|^{3/2})$ 
is irreducible. Hence $\sigma$ is a quotient of $\Theta(\Ind_P^{G_2} (\tau \otimes |\det|^{3/2}))$ and this is what we shall compute.
To that end, since $\Ind_P^{G_2} (\tau \otimes |\det|^{3/2})\cong \Ind_{P}^{G_2} (\tau^{\vee}  \otimes |\det|^{-3/2})$ (using an intertwining operator), 
 we are computing 
\[ 
\Hom_{G_2}(\mathcal{V}_{8}, \Ind_{P}^{G_2} (\tau^{\vee} \otimes |\det|^{-3/2})) \cong \Hom_{\GL_2}(r_{P}(\mathcal{V}_{8}), \tau^{\vee} \otimes |\det|^{-3/2}). 
\] 
Since $-3/2\neq 3/2,7/2$, we see that the top quotient of the filtration of $r_P(\mathcal{V}_{8})$  can be ignored. Since $\tau$ is supercuspidal, the intermediate subquotient can be ignored as well, so 
the computation reduces to the bottom of the filtration of $r_P(\mathcal{V}_{8})$  where it follows at once that 
\[
\Theta(\Ind_P^{G_2} (\tau \otimes |\det|^{3/2}))\cong \Ind_{Q_2}^G( \tau \otimes |\det|^{3/2}).  
\] 
Thus $\sigma$ is a quotient of $\Ind_{Q_2}^G( \tau \otimes |\det|^{3/2})$. This induced representation is a quotient of a standard module for the parabolic subgroup 
contained in $Q_2$ with the Levi $A_{1,\mathrm{short}}$. Thus $\Ind_{Q_2}^G( \tau \otimes |\det|^{3/2})$ has a unique irreducible quotient. 
\end{proof}

\section{Lifting from $F_{4}$ to $\mathrm{Aut}(C)$}\label{F4toAutC}

In this section, we study the theta lift from $F_{4}$ to $\mathrm{Aut}(C)$. Specifically, let $\sigma\in \mathrm{Irr}(G)$. We study $\Theta(\sigma)$ with respect to the minimal representation $(\Pi,\mathcal{V})$ on $\mathcal{G}$, utilizing our results on the lifting from $\mathscr{G}$ to $G$. To begin, we show that the lifting from $G$ to $\mathscr{G}$ has finite length.

\begin{proposition}\label{finlenF4toAutC}
Let $\sigma\in \mathrm{Irr}(G)$. Then $\Theta(\sigma)$ has finite length.
\end{proposition}

\begin{proof} We prove this assuming that $\mathscr{G}=\mathrm{PGL}_{2}(F)$; the non-split case (i.e. when $C$ is anisotropic) is easier. If $\Theta(\sigma)=0$ we are done, so suppose that $\Theta(\sigma)\neq 0$.

Since supercuspidal representations can be split off, the $\mathscr{G}$-representation decomposes as
\begin{equation*}
\Theta(\sigma)=\Theta(\sigma)_{ps}\oplus \Theta(\sigma)_{sc},
\end{equation*}
where $\Theta(\sigma)_{sc}$ is the submodule generated by all of the supercuspidal submodules and $\Theta(\sigma)_{ps}$ is the complement all of whose constituents are constituents of principal series. (When $C$ is anisotropic, $\Theta(\sigma)_{ps}=0$.) To prove the proposition it suffices to show that $\Theta(\sigma)_{ps}$ and $\Theta(\sigma)_{sc}$ have finite length.

We begin with $\Theta(\sigma)_{sc}$. Recall that $\Theta(\sigma)_{sc}$ is completely reducible. Thus if $\Theta(\sigma)_{sc}\neq 0$, then there is a supercuspidal representation $\pi\in \mathrm{Irr}(\mathscr{G})$ such that there is a surjective $\mathscr{G}\times G$ map $\Pi \twoheadrightarrow \pi\otimes \sigma$. By Theorem \ref{SuperCuspLift} part (\ref{SCIrr}), $\Theta(\pi)$ is irreducible, so $\sigma\cong \Theta(\pi)$. Moreover, by Theorem \ref{SuperCuspLift} part (\ref{PartialSmallTheta1-1}) it follows that $\Theta(\sigma)_{sc}\cong \pi$. In particular, $\Theta(\sigma)_{sc}$ has finite length. (This proves the result when $C$ is anisotropic.)

Next we consider $\Theta(\sigma)_{ps}$. Note that if $\rho$ is any smooth representation of $\mathscr{G}$, then $\rho_{ps}$ is of finite length if and only if $\rho_{\overline{\mathscr{U}}}$ is finite dimensional. By Lemma \ref{ThetaDual} $(\Theta(\sigma)_{\overline{\mathscr{U}}})^{*}\cong \mathrm{Hom}_{G}(\Pi_{\overline{\mathscr{U}}},\sigma)$. So we show that $\mathrm{Hom}_{G}(\Pi_{\overline{\mathscr{U}}},\sigma)$ is finite dimensional.

It suffices to analyze the hom-space for each piece of the filtration of $\Pi_{\overline{\mathscr{U}}}$ from Proposition \ref{JMIofMin}. The quotient $\Pi_{\overline{\mathcal{N}}}$ is finite length as a $G$-module by Theorem \ref{PFiltration} and Corollary \ref{pm1F4Theta} (which does not depend on this result). So, $\mathrm{Hom}_{G}(\Pi_{\overline{\mathcal{N}}},\sigma)$ is finite dimensional.

Now we consider the submodule $i_{\mathscr{T}\times Q}^{\mathscr{T}\times G}(C^{\infty}_{c}(F^{\times}))\cong i_{Q}^{G}(C^{\infty}_{c}(F^{\times}))$. Let $L$ be a Levi subgroup of $Q$. By Bernstein's second adjointness, we have
\begin{equation*}
\mathrm{Hom}_{G}(i_{Q}^{G}(C^{\infty}_{c}(F^{\times})),\sigma)\cong \mathrm{Hom}_{L}(C^{\infty}_{c}(F^{\times}),r_{\overline{Q}}(\sigma)).
\end{equation*} 
 The action of $L$ on $C^{\infty}_{c}(F^{\times})$ factors through the fundamental weight $\omega_{4}:L\twoheadrightarrow F^{\times}$. Thus $L$ acts on $C^{\infty}_{c}(F^{\times})$ through the geometric action of $F^{\times}$. Since $\mathrm{dim}(\mathrm{Hom}_{F^{\times}}(C^{\infty}_{c}(F^{\times}),\chi))=1$ for any character $\chi$, it follows that  $\mathrm{dim}(\mathrm{Hom}_{L}(C^{\infty}_{c}(F^{\times}),r_{\overline{Q}}(\sigma)))$ is no larger than the number of one dimensional constituents of $r_{\overline{Q}}(\sigma)$. \end{proof}

In a moment we shall make the computation of $\Hom_G(\Pi_{\mathscr{\overline{U}}} , \sigma)$ more precise, but first note the following corollary: 

\begin{corollary}\label{HDFails} If $\Theta(\sigma)\neq 0$, then $\sigma$ is a quotient of $\Theta(\pi)$ for some $\pi\in\mathrm{Irr}(\mathscr{G})$. 
Moreover $\theta(\sigma)$ is irreducible and $\theta(\sigma_1)\cong \theta(\sigma_2)\neq 0$ implies $\sigma_1\cong\sigma_2$ except in 
one case when $\sigma_1\oplus  \sigma_2$ is the co-socle of $i_{Q}^{G}(|-|^{\frac{5}{2}}\circ \omega_{4})$.  
\end{corollary} 
\begin{proof} Since $\Theta(\sigma)$ has finite length, it has an irreducible quotient $\pi$. Then clearly $\sigma$ is a quotient of $\Theta(\pi)$.  
The other statements are now trivial consequences of what we know about the lift from $\mathscr{G}$.\end{proof} 

\begin{lemma}\label{IrredLiftLemma}
Let $\pi\in \mathrm{Irr}(\mathscr{G})$ such that $\sigma\stackrel{\mathrm{def}}{=}\Theta(\pi)\in\mathrm{Irr}(G)$. Then $\Theta(\sigma)\cong \pi$.
\end{lemma}

\begin{proof} Let $\Psi$ be a rank $3$ character of $\overline{N}$ as in Corollary \ref{Rk3Coin}. We apply $(\overline{N},\Psi)$-coinvariants to the natural surjective map $\mathcal{V}\twoheadrightarrow\Theta(\sigma)\otimes\sigma$ to get a surjective map $\mathcal{V}_{(\overline{N},\Psi)}\twoheadrightarrow\Theta(\sigma)\otimes\sigma_{(\overline{N},\Psi)}$. By Corollary \ref{Rk3Coin} we have $\mathcal{V}_{(\overline{N},\Psi)}\cong C^{\infty}_{c}(\mathscr{G})$ and by Proposition \ref{ThetaPsiCo} we have $\sigma_{(\overline{N},\Psi)}\cong \widetilde{\pi}$. Thus $\Theta(\sigma)\cong \pi$.\end{proof}

\textbf{Remark:} In the previous lemma, the assumption that $\sigma$ is irreducible is required for the definition of $\Theta(\sigma)$.

\begin{theorem}\label{F4toPGL2Irr}
Let $\sigma\in \mathrm{Irr}(G)$ such that $\Theta(\sigma)\neq 0$. Then $\Theta(\sigma)\in\mathrm{Irr}(\mathscr{G})$.
\end{theorem}

\begin{proof} By Proposition \ref{finlenF4toAutC}, $\Theta(\sigma)$ is finite length. So, there exists $\pi\in \mathrm{Irr}(\mathscr{G})$ such that $\Theta(\sigma)\twoheadrightarrow \pi$. If $\Theta(\pi)$ is irreducible, then $\sigma\cong \Theta(\pi)$, and thus Lemma \ref{IrredLiftLemma} implies that $\pi\cong \Theta(\sigma)$. 

It remains to consider the case where $\Theta(\pi)$ is reducible. By Theorem \ref{SuperCuspLift}, part (\ref{SCIrr}), this can occur only if $\mathscr{G}\cong \mathrm{PGL}_{2}(F)$. Moreover, by Proposition \ref{SmallThetaPS}, we see that $\pi$ must be isomorphic to $i_{\overline{\mathscr{B}}}^{\mathscr{G}}(|-|^{s})$, where $s\in \{\frac{5}{2},\frac{11}{2}\}$. Thus $\sigma$ is a quotient of $\Theta(\pi)\cong i_{Q}^{G}(\chi\circ\omega_{4})$, which by Proposition \ref{CJDPS} implies that $\sigma$ is one of three possible representations. When $s=\frac{11}{2}$, then $\sigma$ is the trivial representation; when $s=\frac{5}{2}$, then $\sigma$ is one of the two irreducible representations of the co-socle of $i_{Q}^{G}(|-|^{\frac{5}{2}}\circ\omega_{4})$, which we call $\sigma^{+}$ and $\sigma^{-}$. Moreover, $\Theta(\sigma)$ has the irreducible principal series $i_{\overline{\mathscr{B}}}^{\mathscr{G}}(|-|^{s})$ as a quotient, so $\mathrm{dim}(\Theta(\sigma)_{\overline{\mathscr{U}}})\geq 2$.

From the proof of Proposition \ref{finlenF4toAutC} $\mathrm{dim}(\Theta(\sigma))_{\overline{\mathscr{U}}}$ is less than or equal to $a+b$, where $a=\mathrm{dim}(\mathrm{Hom}_{G}(\mathcal{V}_{\overline{\mathcal{N}}},\sigma))$ and $b$ is the number of constituents of $r_{Q}(\sigma)$ of dimension $1$. In either case, $a=1$ by Theorem \ref{PFiltration} and Corollary \ref{pm1F4Theta} (which does not depend on this result).

Suppose that $\sigma$ is trivial (so $s=\frac{11}{2}$). Then $b=1$ and the result follows in this case.

Suppose that $\sigma\cong \sigma^{\pm}$ (so $s=\frac{5}{2}$). We claim that $b=1$ in this case too, from which the result follows.
The representations $\sigma^+$ and $\sigma^-$ have Iwahori-fixed vectors, and the corresponding Hecke algebra  $H_G$-modules are $E_{\mathcal G}$ and $E_{\mathcal G''}$ in 
\cite[page 640]{L83}.  
On the level of $H_G$- modules, the functor $r_{\overline P}$ correspond to restricting to the Heckle algebra $H_M\subset H_G$. 
 Now it is easy to check that  $E_{\mathcal G}$ and $E_{\mathcal G''}$ embed into $i_{Q}^{G}(|-|^{-\frac{5}{2}}\circ\omega_{4})$, 
giving us the claimed identification with $\sigma^+$ and $\sigma^-$, and 
 that $r_{\overline P}(\sigma^{\pm})$ are of length two, with only one one-dimensional summand, each, as desired. 
\end{proof}


\section{$\mathrm{Spin}(9)$ distinguished representations of $F_{4}$}\label{Spin9}

The objective of this section is to prove a multiplicity one result for $\mathrm{Spin}(9)$-invariant linear functionals and characterize the $\mathrm{Spin}(9)$-distinguished representations of $F_{4}$ as those arising from the theta lift of generic representations on $\mathrm{PGL}_{2}(F)$. We continue to use the notation of Section \ref{JMII}. In particular, $\mathscr{G}=\mathrm{PGL}_{2}(F)$. Let $H=\mathrm{Stab}_{G}(v_{0})\cong \mathrm{Spin}(9,F)$. (Recall Lemma \ref{rk1tr1Trans}.)

\begin{theorem}\label{Spin9Period}
Let $\sigma$ be an irreducible representation of $G$.
 Then the dimension of $\mathrm{Hom}_{H}(\sigma,\mathbb{C})$ is at most $1$. Moreover, $\sigma$ is $H$-distinguished if and only if $\Theta(\sigma)$ is generic. 
\end{theorem}

\begin{proof}
By Lemma \ref{ThetaDual} there is an isomorphism
\begin{equation*}
(\Theta(\sigma)_{(\overline{\mathscr{U}},\psi)})^{*}\cong \mathrm{Hom}_{G}(\mathcal{V}_{(\overline{\mathscr{U}},\psi)},\sigma).
\end{equation*}
By Proposition \ref{MinRepPGL2TwistedJac},
\begin{equation*}
\mathrm{Hom}_{G}(\mathcal{V}_{(\overline{\mathscr{U}},\psi)},{\sigma})\cong \mathrm{Hom}_{G}(ind_{H}^{G}(1),{\sigma}).
\end{equation*}
By taking duals and applying Frobenius reciprocity we have
\begin{equation*}
\mathrm{Hom}_{G}(ind_{H}^{G}(1),{\sigma})\cong \mathrm{Hom}_{H}(\tilde \sigma,\mathbb{C}).
\end{equation*}

By Theorem \ref{F4toPGL2Irr}, $\Theta({\sigma})$ is irreducible, if nonzero. Thus by the multiplicity one theorem for Whittaker functionals, $\mathrm{dim}((\Theta(\sigma)_{(\overline{\mathscr{U}},\psi)})^{*})\leq 1$. Thus $\mathrm{dim}(\mathrm{Hom}_{H}(\tilde \sigma,\mathbb{C}))\leq 1$.

We now need: 

\begin{lemma}\label{SelfDual} If $\Theta(\sigma)\neq 0$, then $\sigma\cong \widetilde \sigma$. 
\end{lemma} 

\begin{proof}
Since $\Theta(\sigma)\neq 0$ there exists $\pi\in \mathrm{Irr}(\mathscr{G})$ such that $\sigma$ is a constituent of $\Theta(\pi)$. 
We prove the result by considering two cases.

First suppose that $\pi$ is a supercuspidal representation of $\mathscr{G}$. Since supercuspidal representation split off, there is a $\mathscr{G}$-module decomposition such that $\mathcal{V}=\mathcal{V}^{\pi}\oplus\mathcal{V}^{\pi,\perp}$, where $\mathcal{V}^{\pi}$ is the maximal $\pi$-isotypic subspace of $\mathcal{V}$ and $\mathcal{V}^{\pi,\perp}$ is the canonical complementary $\mathscr{G}$-submodule. Since the actions of $\mathscr{G}$ and $G$ commute we see that $G$ acts on both $\mathcal{V}^{\pi}$ and $\mathcal{V}^{\pi,\perp}$. Since all of the constituents of $\mathcal{V}^{\pi,\perp}$ are not isomorphic to $\pi$ it follows that the $\mathscr{G}\times G$-module surjection
\begin{equation*}
\mathcal{V}\twoheadrightarrow \pi\otimes \Theta(\pi)
\end{equation*}
is trivial on $\mathcal{V}^{\pi,\perp}$ and so we have a $\mathscr{G}\times G$-module surjection
\begin{equation}\label{piIsotpic}
\mathcal{V}^{\pi}\twoheadrightarrow \pi\otimes \Theta(\pi).
\end{equation}
Since by definition $\pi\otimes \Theta(\pi)$ is the maximal $\pi$-isotypic quotient of $\mathcal{V}$ it follows that (\ref{piIsotpic}) is an isomorphism $\mathcal{V}^{\pi}\cong \pi\otimes \Theta(\pi)$. Since $\mathcal{V}$ is a unitary $\mathcal{G}$-representation and $\mathcal{V}^{\pi}\subseteq \mathcal{V}$ it follows that $\mathcal{V}^{\pi}$ is a unitary $\mathscr{G}\times G$-representation. Thus we have $\mathscr{G}\times G$-module isomorphisms
\begin{equation*}
\widetilde{\pi}\otimes \widetilde{\Theta(\pi)}\cong \widetilde{\mathcal{V}^{\pi}}
\cong \overline{\mathcal{V}^{\pi}}
\cong \mathcal{V}^{\bar{\pi}}
\cong \mathcal{V}^{\widetilde{\pi}}.
\end{equation*}
All irreducible representations of $\mathscr{G}$ are self-dual, so we have $\widetilde{\pi}\cong \pi$. From the above chain of isomorphisms it follows that $\Theta(\pi)$ is self-dual. 

Since $\pi$ is supercuspidal, Theorem \ref{SuperCuspLift} implies $\Theta(\pi)=\sigma$. Thus $\sigma$ is self-dual.

Now suppose that $\pi$ is a constituent of the principal series $i_{\overline{\mathscr{B}}}^{\mathscr{G}}(\chi)$. Then by Theorem \ref{ThetaPS}, $\sigma$ is a constituent of $i_{Q}^{G}(\chi\circ \omega_{4})$. 

We claim that all of the constituents of $i_{Q}^{G}(\chi\circ \omega_{4})$ are self-dual. By Choi-Jantzen \cite{CJ10}, Theorem 6.1, the length of $i_{Q}^{G}(\chi\circ \omega_{4})$ is less than 3. In each of the following three cases we use that there is a nonzero intertwining operator $i_{Q}^{G}(\chi^{\pm1}\circ \omega_{4})\rightarrow i_{Q}^{G}(\chi^{\mp1}\circ \omega_{4})\cong \widetilde{i_{Q}^{G}}(\chi^{\pm1}\circ \omega_{4})$. 

When $i_{Q}^{G}(\chi\circ \omega_{4})$ is irreducible we are done. When $i_{Q}^{G}(\chi\circ \omega_{4})$ has length $2$, then \cite[Theorem 6.1,1.]{CJ10} implies that $i_{Q}^{G}(\chi^{\pm1}\circ \omega_{4})$ has a unique irreducible sub and a unique irreducible quotient, which are distinct. Thus the nonzero intertwining operators  imply the self-duality of the irreducible constituents of $i_{Q}^{G}(\chi^{\pm1}\circ \omega_{4})$.

When $i_{Q}^{G}(\chi^{\pm1}\circ \omega_{4})$ has length $3$, then $\chi=|-|^{\pm\frac{5}{2}}$. In this case, $i_{Q}^{G}(|-|^{-\frac{5}{2}}\circ \omega_{4})$ has a unique irreducible quotient, and the intertwining operator shows that it is self dual. There is also a decomposible submodule with two distinct constituents, call them $\sigma^{+}$ and $\sigma^{-}$. Using the $\mu_{2}\times G$ dual pair considered in Section \ref{QuadCase} we can show that $\sigma^{+}$ and $\sigma^{-}$ are self-dual. Specifically, $\mathcal{V}_{6}$ the minimal representation of $E_{6}$ decomposes under the action of $\mu_{2}\times G$ as $\mathcal{V}_{6}\cong \sigma^{+}\oplus \sigma^{-}$ (Theorem \ref{E6Theta}). Now if $I$ is an Iwahori subgroup of $G$, then $\mathrm{dim}((\sigma^{+})^{I})=5$ and $\mathrm{dim}((\sigma^{-})^{I})=2$. Thus $\sigma^{+}$ and $\sigma^{-}$ are self-dual.
\end{proof}
The lemma completes the proof of the theorem.
\end{proof}

\textbf{Remark:} Using the $\mathrm{PGL}_{2}\times F_{4}$ theta correspondence we can construct representations of $F_{4}$ that are $\mathrm{Spin}(9)$-relatively supercuspidal (see \cite{M18} for definition), but not supercuspidal. Let $\pi\in\mathrm{Irr}(\mathscr{G})$ be supercuspidal. As in Lemma \ref{SelfDual}, we have an embedding of $\mathscr{G}\times G$-modules $\pi\otimes \Theta(\pi)\hookrightarrow \mathcal{V}$. By taking $(\overline{\mathscr{U}},\Psi)$-coinvariants and applying Proposition \ref{MinRepPGL2TwistedJac} we get an embedding of $G$-modules $\Theta(\pi)\hookrightarrow C_{c}^{\infty}(H\backslash G)$, where $H\cong \mathrm{Spin}(9)$. Thus $\Theta(\pi)$ is $H$-relatively supercuspidal. But, by Proposition \ref{ThetaSCasQuotient} we know that $\Theta(\pi)$ is not supercuspidal.

\section{Dual pair $\mu_{2}\times F_{4}\subset E_{6}$}\label{QuadCase}

In this section, we study the theta lift associated to the dual pair $\mu_{2}\times F_{4}\subset E_{6}$, where $E_{6}$ is of adjoint form with two connected components where the action of the nontrivial component is through the outer automorphism of $E_{6}$. This situation arises from the construction of Subsection \ref{DualPairs} by taking $C$ to be a quadratic composition algebra.

The analysis of this case is similar to and simpler than the case of $E_{7}$ considered in Section \ref{E7SuperCusp}, so we will be brief. We note that the results of Subsection \ref{mu2F41} could have been proved after Section \ref{JMI}, but our proof of the result of Subsection \ref{mu2F42} utilizes Theorem \ref{ThetaPS}.

Let $\mathcal{G}$ be the $F$-points of the adjoint form of $E_{6}$ constructed using the quadratic composition algebra $C$ with two connected components where the nontrivial component acts through the outer automorphism associated with a choice of simple roots $\Delta$. Let $(\Pi,\mathcal{V})$ be the minimal representation of $\mathcal{G}$. Let $G$ be the fixed points in the identity component of $\mathcal{G}$ under the action of the outer automorphism. Then if we identify $\mu_{2}$ with the subgroup of $\mathcal{G}$ generated by the outer automorphism, then $\mu_{2}\times G\subset \mathcal{G}$ is a dual pair.

Let $\tau^{+}$ and $\tau^{-}$ be the trivial and nontrivial characters of $\mu_{2}$, respectively. There is a surjective map $\mathcal{V}\twoheadrightarrow \tau^{\pm}\otimes \Theta(\tau^{\pm})$. The goal of this section is to compute $\Theta^{\pm}=\Theta(\tau^{\pm})$.

\subsection{Lifting from $\mu_{2}$ to $F_{4}$}\label{mu2F41}

\begin{theorem}
The $G$-module $\Theta^{\pm}$ is irreducible and $\Theta^{+}\ncong\Theta^{-}$.
\end{theorem}

\begin{proof} First, we show that $\mathrm{FJ}(\Theta^{\pm})\cong \omega_{\psi}^{\pm}$, where $\omega_{\psi}^{+}$ and $\omega_{\psi}^{-}$ are the even and odd Weil representation of $\mathrm{Sp}(6,F)$, respectively. This follows from the analogs of Lemmas \ref{MinZTwist} and \ref{MinZTwist2} and Propositions \ref{FJMin} and \ref{FJThetaSubQuo}. The main difference in this case is that the dimension of $N^{\perp}/Z$ is 6, so $\omega_{\psi}$ is the Weil representation of $\mathrm{Sp}(6,F)$, and $\mathrm{Aut}(C)\cong\mu_{2}$ is the full orthogonal group $\mathrm{O}(C^{0})$ (as opposed to $\mathrm{SO}(C^{0})$). Since $\omega_{\psi}^{+}\ncong \omega_{\psi}^{-}$ it follows that $\Theta^{+}\ncong\Theta^{-}$.

From this we also see that $\Theta^{\pm}$ has exactly one nontrivial constituent, since the Fourier-Jacobi functor is exact and only kills the trivial representation.

Second, we show that $\Theta^{\pm}$ does not contain the trivial representation as a constituent. If it does, then $\mathcal{V}_{\overline N}$ contains the trivial 
representation of $M$ as a constituent. However, by Proposition \ref{Rk0Coin} part (\ref{Rk0Aniso}) and Theorem \ref{HeisJacThm}, 
\begin{equation*}
\mathcal{V}_{\overline N} \cong \mathcal{V}(\mathcal{M})\otimes  |\mathrm{det}|^{2/20} \oplus  \chi_C |\mathrm{det}|^{4/20}.
\end{equation*}
We see that the center of $M$ acts by non-trival characters on the two summands of $\mathcal{V}_{\overline N}$, thus $\mathcal{V}_{\overline N}$ cannot contain the trivial representation of $M$ as a constituent. Therefore $\mathcal{V}$ cannot contain the trivial representation of $G$ as a constituent.

From this it follows that $\Theta^{\pm}$ is irreducible.\end{proof}

\begin{corollary}\label{pm1F4Theta}
As a $\mu_{2}\times G$-module, $\mathcal{V}\cong \Theta^{+}\oplus \Theta^{-}$.
\end{corollary}

\subsection{$C=F\oplus F$}\label{mu2F42}

In this section we use our results on the $\mathrm{PGL}(2)\times F_{4}\subset E_{7}$ dual pair to make the lift of $\mu_{2}$ to $F_{4}$ induced from the split form of $E_{6}$ explicit.

Throughout we use the notation of Section \ref{JMII} and we write $\mathcal{V}_{n}$ for the minimal representation of $E_{n}$.

We apply $\overline{\mathscr{U}}$-coinvariants to sequence (\ref{MinSeq}) to get the surjective $\mathscr{T}\times G$-module map

\begin{equation*}
(\mathcal{V}_{7})_{\overline{\mathscr{U}}}\twoheadrightarrow (\mathcal{V}_{7})_{\overline{\mathcal{N}}}\cong \mathcal{V}_{6}\otimes |-|^{3}\oplus |-|^{6}\twoheadrightarrow \Theta^{\pm}\otimes |-|^{3}.
\end{equation*}
Thus Frobenius reciprocity with respect to $\overline{\mathscr{B}}\subset \mathscr{G}$ yields a nonzero $\mathscr{G}\times G$-module map

\begin{equation*}
\mathcal{V}_{7}\rightarrow \mathrm{Ind}_{\overline{\mathscr{B}}}^{\mathscr{G}}(\Theta^{\pm}\otimes|-|^{3}). 
\end{equation*}
Since $\mathscr{T}$ is the center of $\mathcal{M}$ it acts trivially on $\mathcal{V}_{6}$, thus $\mathrm{Ind}_{\overline{\mathscr{B}}}^{\mathscr{G}}(\Theta^{\pm}\otimes|-|^{3})\cong i_{\overline{\mathscr{B}}}^{\mathscr{G}}(|-|^{\frac{5}{2}})\otimes \Theta^{\pm}$. Note that $i_{\overline{\mathscr{B}}}^{\mathscr{G}}(|-|^{\frac{5}{2}})\otimes \Theta^{\pm}$ is an irreducible $\mathscr{G}\times G$-module. Thus by Theorem \ref{ThetaPS}, there is a surjective $G$-module map

\begin{equation*}
i_{Q}^{G}(|-|^{\frac{5}{2}}\circ\omega_{4})\cong\Theta(i_{\overline{\mathscr{B}}}^{\mathscr{G}}(|-|^{\frac{5}{2}}))\twoheadrightarrow \Theta^{\pm}.
\end{equation*}

By applying Proposition \ref{CJDPS} we get the following theorem.

\begin{theorem}\label{E6Theta}
There is a bijection between the irreducible $G$-modules $\{\Theta^{+},\Theta^{-}\}$ and the two irreducible summands of the unique semisimple quotient of 
$i_{Q}^{G}(|-|^{\frac{5}{2}}\circ\omega_{4})$.  
\end{theorem}

\section{Acknowledgments} Some of the key ideas used in this work were conceived in conversations with Wee Teck Gan, 
during a Research in Teams event at the Erwing Schroedinger Institute in Vienna, in April 2022. 
This work was finished at the National University of Singapore in October 2023. The authors would like to thank these institutions for hospitality and 
Wee Teck Gan for help and support throughout years. 
The second author has been supported by a gift from the Simons Foundation No. 946504.

\end{document}